\documentclass[11pt,reqno,bold]{paper}
\usepackage{geometry}
\usepackage{amsmath}
\usepackage{amssymb}
\usepackage{amsfonts}
\usepackage{mathrsfs}
\usepackage{asymptote}
\usepackage{caption}
\usepackage{subcaption}

\usepackage{todonotes}

\usepackage[round]{natbib}
\usepackage[algoruled,noline]{algorithm2e}
\usepackage{enumerate}

\usepackage{amsthm}

\newtheorem{theorem}{Theorem}[section]

\newtheorem{lemma}[theorem]{Lemma}
\newtheorem{corollary}[theorem]{Corollary}
\theoremstyle{definition}
\newtheorem{exampl}[theorem]{Example}
\theoremstyle{remark}
\newtheorem{remark}[theorem]{Remark}

\usepackage{color}
\definecolor{darkblue}{rgb}{0.0,0.0,0.4}
\usepackage[pdfauthor={Tsogtgerel Gantumur},pdftitle={Convergence rates of adaptive methods},pdfsubject={adaptive methods, convergence rates, approximation spaces, Besov spaces, multilevel approximation},colorlinks,citecolor=darkblue,linkcolor=darkblue,urlcolor=darkblue]{hyperref}

\newcommand{\N}{\mathbb{N}}
\newcommand{\Z}{\mathbb{Z}}
\newcommand{\R}{\mathbb{R}}
\newcommand{\Pol}{\mathbb{P}}
\newcommand{\tstA}{\mathscr{A}}
\newcommand{\tstP}{\mathscr{P}}
\newcommand{\tstD}{\mathscr{D}}

\newcommand{\eps}{\varepsilon}

\newcommand{\diam}{\mathrm{diam}}
\newcommand{\osc}{\mathrm{osc}}

\newcommand{\exd}{\mathrm{d}}
\newcommand{\supp}{\mathrm{supp}}

\newcommand{\refine}{\mathrm{refine}}

\newcommand{\Leb}[2][]{L_{#2}^{#1}}
\newcommand{\leb}[2][]{\ell_{#2}^{#1}}

\title{Convergence rates of adaptive methods, Besov spaces, and multilevel approximation}
\author{Tsogtgerel Gantumur}
\institution{McGill University}
\date{\today}                                           

\begin{document}

\maketitle

\begin{abstract}
This paper concerns characterizations of approximation classes associated to adaptive finite element methods with isotropic $h$-refinements.
It is known from the seminal work of Binev, Dahmen, DeVore and Petrushev that such classes are related to Besov spaces.
The range of parameters for which the inverse embedding results hold is rather limited,
and recently, Gaspoz and Morin have shown, among other things, that this limitation disappears if we replace Besov spaces by suitable approximation spaces
associated to finite element approximation from uniformly refined triangulations.
We call the latter spaces {\em multievel approximation spaces}, 
and argue that these spaces are placed naturally halfway between adaptive approximation classes and Besov spaces,
in the sense that it is more natural to relate multilevel approximation spaces with either Besov spaces or adaptive approximation classes,
than to go directly from adaptive approximation classes to Besov spaces.
In particular, we prove embeddings of multilevel approximation spaces into adaptive approximation classes,
complementing the inverse embedding theorems of Gaspoz and Morin.

Furthermore, in the present paper,
we initiate a theoretical study of adaptive approximation classes that are defined using a modified notion of error, the so-called {\em total error}, which is the energy error plus an oscillation term. Such approximation classes have recently been shown to arise naturally in the analysis of adaptive algorithms.
We first develop a sufficiently general approximation theory framework to handle such modifications, and then apply the abstract theory to second order elliptic problems discretized by Lagrange finite elements, resulting in characterizations of modified approximation classes in terms of memberships of the problem solution and data into certain approximation spaces,
which are in turn related to Besov spaces.
Finally, it should be noted that throughout the paper we paid equal attention to both conforming and nonconforming triangulations.
\end{abstract}

\tableofcontents

\section{Introduction}

Among the most important achievements in theoretical numerical analysis during the last decade was
the development of mathematical techniques for analyzing the performance of adaptive finite element methods.
A crucial notion in this theory is that of {\em approximation classes}, which we discuss here in a simple but very paradigmatic setting.
Given a polygonal domain $\Omega\in\R^2$, a conforming triangulation $P_0$ of $\Omega$, and a number $s>0$, 
we say that a function $u$ on $\Omega$ belongs to the approximation class $\tstA^s$
if for each $N$, there is a conforming triangulation $P$ of $\Omega$ with at most $N$ triangles,
such that $P$ is obtained by a sequence of newest vertex bisections from $P_0$,
and that $u$ can be approximated by a continuous piecewise affine function subordinate to $P$ 
with the error bounded by $cN^{-s}$, where $c=c(u,P_0,s)\geq0$ is a constant independent of $N$.
In a typical setting, the error is measured in the $H^1$-norm, which is the natural energy norm for second order elliptic problems.
To reiterate and to remove any ambiguities, we say that $u\in H^1(\Omega)$ belongs to $\tstA^s$ if 
\begin{equation}\label{e:approx-class-std-intro}
\min_{\{P\in\tstP:\#P\leq N\}} \inf_{v\in S_P} \|u-v\|_{H^1} \leq c N^{-s},
\end{equation}
for all $N\geq\#P_0$ and for some constant $c$,
where $\tstP$ is the set of conforming triangulations of $\Omega$ that are obtained by a sequence of newest vertex bisections from $P_0$,
and $S_P$ is the space of continuous piecewise affine functions subordinate to the triangulation $P$.

Approximation classes can be used to reveal a theoretical barrier on any procedure that is designed to approximate $u$
by means of piecewise polynomials and a fixed refinement rule such as the newest vertex bisection. 
Suppose that we start with the initial triangulation $P_0$, and generate a sequence of conforming triangulations by using newest vertex bisections.
Suppose also that we are trying to capture the function $u$ by using continuous piecewise linear functions subordinate to the generated triangulations.
Finally, assume that $u\in\tstA^s$ but $u\not\in\tstA^\sigma$ for any $\sigma>s$.
Then as far as the exponent $\sigma$ in $cN^{-\sigma}$ is concerned, 
it is obvious that the best asymptotic bound on the error we can hope for is $cN^{-s}$, where $N$ is the number of triangles.
Now supposing that $u$ is given as the solution of a boundary value problem,
a natural question is if this convergence rate can be achieved by any practical algorithm,
and it was answered in the seminal works of \cite*{BDD04} and \cite*{Stev07}: These papers established that the convergence rates of certain adaptive finite element methods
are optimal, in the sense that if $u\in\tstA^s$ for some $s>0$, then the method converges with the rate not slower than $s$.
One must mention the earlier developments \cite*{Dorf96,MNS00,CDD01,GHS07}, which paved the way for the final achievement.

Having established that the smallest approximation class $\tstA^s$ in which the solution $u$ belongs to essentially determines how fast 
adaptive finite element methods converge, the next issue is to determine how large these classes are 
and if the solution to a typical boundary value problem would belong to an $\tstA^s$ with large $s$.
In particular, one wants to compare the performance of adaptive methods with that of non-adaptive ones.
A first step towards addressing this issue is to characterize the approximation classes in terms of classical smoothness spaces,
and the main work in this direction so far appeared is \cite*{BDDP02}, which, upon tailoring to our situation and a slight simplification,
tells that $B^{\alpha}_{p,p}\subset\tstA^s\subset B^{\sigma}_{p,p}$
for $\frac2p=\sigma<1+\frac1p$ and $\sigma<\alpha<\max\{2,1+\frac1p\}$ with $s=\frac{\alpha-1}2$.
Here $B^\alpha_{p,q}$ are the standard Besov spaces defined on $\Omega$.
This result has recently been generalized to higher order Lagrange finite elements by \cite{GM13}.
In particular, they show that the direct embedding $B^{\alpha}_{p,p}\subset\tstA^s$ holds in the larger range $\sigma<\alpha<m+\max\{1,\frac1p\}$,
where $m$ is the polynomial degree of the finite element space, see Figure \ref{f:direct-intro}(a).
However, the restriction $\sigma<1+\frac1p$ on the inverse embedding $\tstA^s\subset B^{\sigma}_{p,p}$ cannot be removed,
since for instance, any finite element function whose derivative is discontinuous cannot be in $B^{\sigma}_{p,p}$ if $\sigma\geq1+\frac1p$ and $p<\infty$.
To get around this problem, Gaspoz and Morin proposed to replace the Besov space $B^{\sigma}_{p,p}$ by the approximation space $A^\sigma_{p,p}$
associated to uniform refinements\footnote{This space is denoted by $\hat B^\sigma_{p,p}$ in \cite{GM13}. In the present paper we are adopting the notation of \cite{Osw94}.}.
We call the spaces $A^\sigma_{p,p}$ {\em multilevel approximation spaces}, and their definition will be given in Subsection \ref{ss:multilevel}.
For the purposes of this introduction, and roughly speaking, the space $A^\sigma_{p,p}$ is the collection of functions $u\in\Leb{p}$ for which
\begin{equation}
\inf_{v\in S_{P_k}} \|u-v\|_{\Leb{p}} \leq c h_k^{\sigma},
\end{equation}
where $\{P_k\}\subset\tstP$ is a sequence of triangulations such that $P_{k+1}$ is the uniform refinement of $P_k$,
and $h_k$ is the diameter of a typical triangle in $P_k$.
Note for instance that finite element functions are in every $A^\sigma_{p,p}$.
With the multilevel approximation spaces at hand, the inverse embedding $\tstA^s\subset A^{\sigma}_{p,p}$ is recovered for all $\sigma\leq\frac2p$.

In this paper, we prove the direct embedding $A^{\alpha}_{p,p}\subset\tstA^s$, 
so that the existing situation $B^{\alpha}_{p,p}\subset\tstA^s\subset A^{\sigma}_{p,p}$ is improved to $A^{\alpha}_{p,p}\subset\tstA^s\subset A^{\sigma}_{p,p}$.
It is a genuine improvement, since $A^{\alpha}_{p,p}(\Omega)\supsetneq B^{\alpha}_{p,p}(\Omega)$ for $\alpha\geq1+\frac1p$.
Moreover, as one stays entirely within an approximation theory framework, one can argue that the link between $\tstA^s$ and $A^{\alpha}_{p,p}$ is more natural
than the link between $\tstA^s$ and $B^{\alpha}_{p,p}$.
Once the link between $\tstA^s$ and $A^{\alpha}_{p,p}$ has been established, one can then invoke the well known relationships between $A^{\alpha}_{p,p}$ and $B^{\alpha}_{p,p}$.
It seems that this two step process offers more insight into the underlying phenomenon.
We also remark that while the existing results are only for the newest vertex bisection procedure and conforming triangulations, 
we deal with possibly nonconforming triangulations, and therefore are able to handle the red refinement procedure, 
as well as newest vertex bisections without the conformity requirement.

\begin{figure}[ht]
\centering
\begin{subfigure}{0.45\textwidth}
\includegraphics[width=0.8\textwidth]{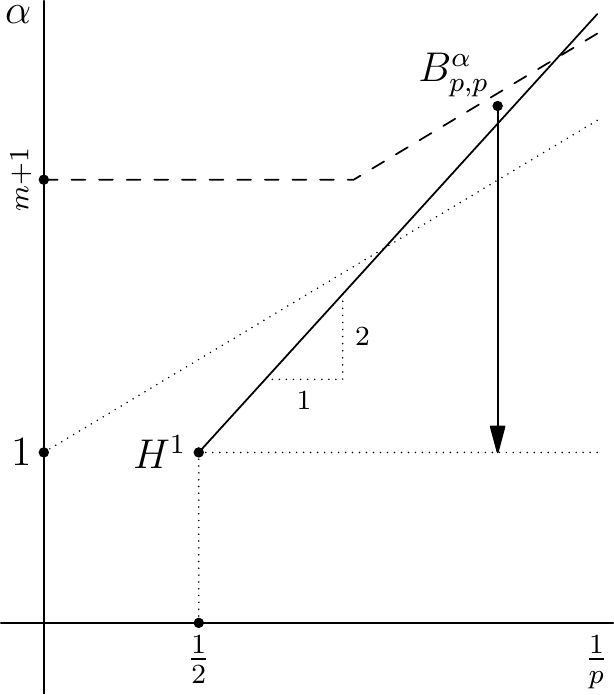}
\subcaption{If the space $B^\alpha_{p,p}$ is located strictly above the solid line and below the dashed line, 
then $B^\alpha_{p,p}\subset\tstA^s$ with $s=\frac{\alpha-1}2$.
The inverse embeddings $\tstA^{s+\eps}\subset B^\alpha_{p,p}$ hold on the solid line and below the (slanted) dotted line.}
\end{subfigure}
\qquad
\begin{subfigure}{0.45\textwidth}
\includegraphics[width=0.8\textwidth]{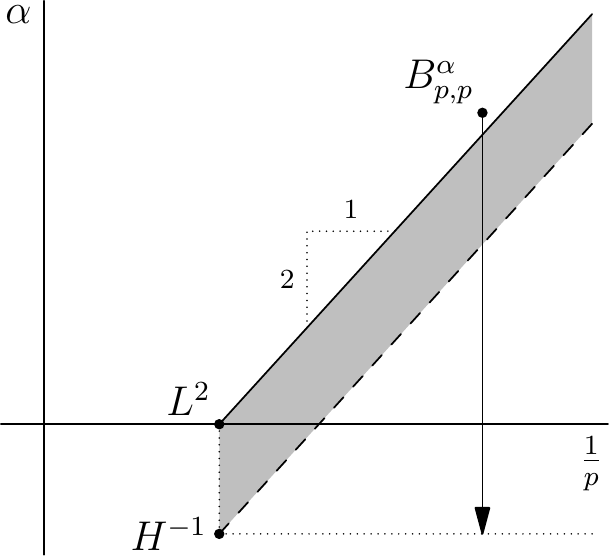}
\subcaption{If the space $B^\alpha_{p,p}$ is located above or on the solid line,
and if $u\in\tstA^s$ and $\Delta u \in B^\alpha_{p,p}$ with $s=\frac{\alpha+1}2$,
then $u\in \tstA^s_*$.
It is as if the approximation of $\Delta u$ is taking place in $H^{-1}$, 
with the proviso that the shaded area is excluded from all considerations.}
\end{subfigure}
\caption{Illustration of various embeddings.
The point $(\frac1p,\alpha)$ represents the space $B^\alpha_{p,p}$.}
\label{f:direct-intro}
\end{figure}

The approximation classes $\tstA^s$ defined by \eqref{e:approx-class-std-intro} are associated to measuring the error of an approximation in the $H^1$-norm.
Of course, this can be generalized to other function space norms, such as $\Leb{p}$ and $B^\alpha_{p,p}$, which we will consider in Section \ref{s:Lagrange}.
However, we will not stop there, and consider more general approximation classes corresponding to ways of measuring the error between a general function $u$ and a discrete function $v\in S_P$
by a quantity $\rho(u,v,P)$ that may depend on the triangulation $P$ and is required to make sense merely for discrete functions $v\in S_P$. 
An example of such an error measure is 
\begin{equation}\label{e:total-error-intro}
\rho(u,v,P)
= \left( \|u-v\|_{H^1}^2 
+ \sum_{\tau\in P} (\diam\,\tau)^{2}\|f-\Pi_\tau f\|_{\Leb{2}(\tau)}^2 \right)^{\frac12} ,
\end{equation}
where $f=\Delta u$,
and $\Pi_\tau:\Leb{2}(\tau)\to\Pol_{d}$ is the $\Leb{2}(\tau)$-orthogonal projection onto $\Pol_{d}$,
the space of polynomials of degree not exceeding $d$.
It has been shown in \cite*{CKNS08} that if the solution $u$ of the boundary value problem
\begin{equation}\label{e:bvp-intro}
\Delta u = f
\quad\textrm{in}\,\,\Omega\qquad\textrm{and}\qquad
u|_\Omega=0 ,
\end{equation}
satisfies 
\begin{equation}\label{e:approx-class-gen-intro}
\min_{\{P\in\tstP:\#P\leq N\}} \inf_{v\in S_P} \rho(u,v,P) \leq c N^{-s},
\end{equation}
for all $N\geq\#P_0$ and for some constants $c$ and $s>0$,
then a typical adaptive finite element method for solving \eqref{e:bvp-intro} converges with the rate not slower than $s$.
Moreover, there are good reasons to consider that the approximation classes $\tstA^s_*$ defined by the condition \eqref{e:approx-class-gen-intro}
are more attuned to certain practical adaptive finite element methods than the standard approximation classes $\tstA^s$ defined by \eqref{e:approx-class-std-intro}, see Section \ref{s:2nd-order}.
Obviously, we have $\tstA^s_*\subset\tstA^s$ but we cannot expect the inclusion $\tstA^s\subset\tstA^s_*$ to hold in general.
In \cite{CKNS08}, an effective characterization of $\tstA^s_*$ was announced as an important pending issue.

In the present paper, we establish a characterization of $\tstA^s_*$ in terms of memberships of $u$ and $f=\Delta u$ into 
suitable approximation spaces, which in turn are related to Besov spaces.
For instance, we show that if $u\in\tstA^s$ and $f\in B^\alpha_{p,p}$ with $\frac\alpha2\geq\frac1p-\frac12$ and $s=\frac{\alpha+1}2$, then $u\in\tstA^s_*$, see Figure \ref{f:direct-intro}(b).
Note that the approximation rate $s=\frac{\alpha+1}2$ is as if we were approximating $f$ in the $H^{-1}$-norm, which is illustrated by the arrow downwards.
However, the parameters must satisfy $\frac\alpha2\geq\frac1p-\frac12$ (above or on the solid line), which is more restrictive compared to $\frac{\alpha+1}2>\frac1p-\frac12$ (above the dashed line),
the latter being the condition we would expect if the approximation was indeed taking place in $H^{-1}$.
This situation cannot be improved in the sense that if $\frac\alpha2<\frac1p-\frac12$ then $B^\alpha_{p,p}\not\subset \Leb{2}$, 
hence the quantity \eqref{e:total-error-intro} would be infinite in general for $f\in B^\alpha_{p,p}$.

At this point, the reader might be wondering if we can deduce $f\in B^\alpha_{p,p}$ from $u\in\tstA^s$.
If so, everything would follow from the single assumption $u\in\tstA^s$,
which would then render the theory more in line with the traditional results.
However, $u$ can be in a space slightly smaller than $H^1$, and in this case we cannot guarantee $f\in L^2$, and hence the quantity \eqref{e:total-error-intro} would not be defined.
On the other hand, there are standard examples where $f$ is smooth but $u$ is barely in $H^{1+\eps}$ for a small $\eps$.
Furthermore, since we are solving a PDE, and $f$ is ``given'', there would generally be much more information available about $f$ than about $u$, 
and so it is not an urgent matter to deduce the regularity of $f$ from that of $u$.

The results of Section \ref{s:2nd-order} and some of the results of Section \ref{s:Lagrange} are proved by invoking abstract theorems that are established in Section \ref{s:general}.
These theorems extend some of the standard results from approximation theory to deal with generalized approximation classes such as $\tstA^s_*$.
We decided to consider a fairly general setting in the hope that the theorems will be used for establishing characterizations of other approximation classes.
For example, adaptive boundary element methods and adaptive approximation in finite element exterior calculus seem to be amenable to our abstract framework,
although checking the details poses some technical challenges.

This paper is organized as follows.
In Section \ref{s:general}, we introduce an abstract framework that is more general than usually considered in approximation theory of finite element methods,
and collect some theorems that can be used to prove embedding theorems between adaptive approximation classes and other function spaces.
In Section \ref{s:Lagrange}, we recall some standard results on multilevel approximation spaces and their relationships with Besov spaces,
and then prove direct embedding theorems between multilevel approximation spaces and adaptive approximation classes.
The main results of this section are Theorem \ref{t:direct-Lp}, Theorem \ref{t:direct-App}, and Theorem \ref{t:direct-Lp-disc}.
Finally, in Section \ref{s:2nd-order}, 
we investigate approximation classes associated to certain adaptive finite element methods for variable coefficient second order boundary value problems.
We emphasize that the actual results in Section \ref{s:2nd-order} are in terms of the approximation spaces that are studied in Section \ref{s:Lagrange},
and in order to relate them to Besov spaces, one has to appeal to Section \ref{s:Lagrange}.

\section{General theorems}
\label{s:general}

\subsection{The setup}
\label{ss:setup}

Let $M$ be an $n$-dimensional topological manifold, equipped with a compatible measure, in the sense that all Borel sets are measurable.
What we have in mind here is $M=\R^n$ with the Lebesgue measure on it,
or a piecewise smooth surface $M\subset\R^N$ with its canonical Hausdorff measure.
With $\Omega\subset M$ a bounded domain,
we consider a class of partitions (triangulations) of $\Omega$, and finite element type functions defined over those partitions.
Ultimately, we are interested in characterizing those functions on $\Omega$ that can be well approximated by such finite element type functions.
In order to make these concepts precise, we will use in this section a fairly abstract setting, 
which we believe to be a good compromise between generality and readability.

By a {\em partition} of $\Omega$ we understand a collection $P$ of finitely many disjoint open subsets of $\Omega$, 
satisfying $\overline\Omega = \bigcup_{\tau\in P}\overline\tau$.
We assume that a set $\tstP$ of partitions of $\Omega$ is given, which we call the set of {\em admissible partitions}.
For simplicity, we will assume that for any $k\in\N$ the set $\{P\in\tstP:\#P\leq k\}$ is finite.
In practice, $\tstP$ would be, for instance, the set of all {\em conforming} triangulations obtained from a fixed initial triangulation $P_0$ 
by repeated applications of the newest vertex bisection procedure.
Another class of important examples arises when we want to allow partitions with hanging nodes.
In this case, an admissibility criterion on a partition has been discussed in \cite{BN10}.
Here and in the following, we often write triangles and edges et cetera to mean $n$-simplices and $(n-1)$-dimensional faces et cetera,
which seems to improve readability.
Hence note that the use of a two dimensional language does {\em not} mean that the results we discuss are valid only in two dimensions.

We will assume the existence of a {\em refinement procedure} satisfying certain requirements.
Given a partition $P\in\tstP$ and a set $R\subset P$ of its elements,
the refinement procedure produces $P'\in\tstP$,
such that $P\setminus P'\supset R$, i.e., the elements in $R$ are refined at least once.
Let us denote it by $P'=\refine(P,R)$.
In practice, this is implemented by a usual naive refinement possibly producing a non-admissible partition,
followed by a so-called {\em completion} procedure.
We assume the existence of a constant $\lambda>1$ such that $|\tau|\leq\lambda^{-n}|\sigma|$ for 
all $\tau\in P'$ and $\sigma\in R$ with $\tau\cap\sigma\neq\varnothing$.
Note that we have $\lambda=2$ for red refinements, and $\lambda=\sqrt[n]2$ for the newest vertex bisection.
Moreover, we assume the following on the efficiency of the refinement procedure:
If $\{P_k\}\subset\tstP$ and $\{R_k\}$ are sequences such that $P_{k+1}=\refine(P_k,R_k)$
and $R_k\subset P_k$ for $k=0,1,\ldots$,
then
\begin{equation}\label{e:complete}
\#P_k-\#P_0\lesssim \sum_{m=0}^{k-1} \#R_m,
\qquad k=1,2,\ldots.
\end{equation}
Here and in what follows, we shall often dispense with giving explicit names to constants, and use the Vinogradov-style notation $X\lesssim Y$, 
which means that $X\leq C\cdot Y$ with some constant $C$ that is allowed to
depend only on $\tstP$ and (the geometry of) the domain $\Omega$.
Assumption \eqref{e:complete} is justified for newest vertex bisection algorithm in
\cite{BDD04,Stev08},
and the red refinement rule is treated in \cite{BN10}.

Next, we shall introduce an abstraction of finite element spaces.
To this end, we assume that there is a quasi-Banach space $X_0$,
and for each $P\in\tstP$, there is a nontrivial, finite dimensional subspace $S_P\subset X_0$.
The space $X_0$ models the function space over $\Omega$ in which the approximation takes place,
such as $X_0=H^t(\Omega)$ and $X_0=\Leb{p}(\Omega)$.
The spaces $S_P$ are, as the reader might have guessed, models of finite element spaces, from which we approximate 
general functions in $X_0$.
Obviously, a natural notion of error between an element $u\in X_0$ and its approximation $v\in S_P$ is the quasi-norm $\|u-v\|_{X_0}$.
However, we need a bit more flexibility in how to measure such errors, and so
we suppose that there is a function $\rho(u,v,P)\in[0,\infty]$ defined for $u\in X_0$, $v\in S_P$, and $P\in\tstP$.
Note that this error measure, which we call a {\em distance function}, can depend on the partition $P$, and it is only required to make sense for functions $v\in S_P$.
We allow the value $\rho=\infty$ to leave open the possibility that for some $u\in X_0$ we have $\rho(u,\cdot,\cdot)=\infty$.
The most important distance function is still $\rho(u,v,P)=\|u-v\|_{X_0}$, but other examples will appear later in the paper,
see e.g., Example \ref{eg:Poisson}, Subsection \ref{ss:disc} and Section \ref{s:2nd-order}.

Given $u\in X_0$ and $P\in\tstP$, we let
\begin{equation}\label{e:E-lin}
E(u,S_P)_\rho = \inf_{v\in S_P} \rho(u,v,P),
\end{equation}
which is the error of a best approximation of $u$ from $S_P$.
Furthermore, we introduce
\begin{equation}\label{e:E-nonlin}
E_k(u)_\rho = \inf_{\{P\in\tstP:\#P\leq2^kN\}} E(u,S_P)_\rho,
\end{equation}
for $u\in X_0$ and $k\in\N$, with the constant $N$ chosen sufficiently large in order to ensure that the set $\{P\in\tstP:\#P\leq2N\}$ is nonempty.
In a certain sense, $E_k(u)_\rho$ is the best approximation error when one tries to approximate $u$ within the budget of $2^kN$ triangles.
Finally, we define the main object of our study, the {\em (adaptive) approximation class}
\begin{equation}\label{e:A-rho-def}
\tstA^s_q(\rho) = \tstA^s_q(\rho,\tstP,\{S_P\}) = \{ u\in X_0 : |u|_{\tstA^s_q(\rho)}<\infty \} ,
\end{equation}
where $s>0$ and $0<q\leq\infty$ are parameters, and
\begin{equation}\label{e:apps-inf}
|u|_{\tstA^s_q(\rho)} = \|(2^{ks}E_k(u)_\rho)_{k\in\N}\|_{\leb{q}} ,
\qquad u\in X_0.
\end{equation}
In the following, we will use the abbreviation $\tstA^s(\rho)=\tstA^s_\infty(\rho)$.
Note that $u\in\tstA^s_q(\rho)$ implies $E_k(u)_\rho\leq c2^{-ks}$ for all $k$ and for some constant $c$, and these two conditions are equivalent if $q=\infty$.
We have $\tstA^s_q(\rho)\subset\tstA^s_r(\rho)$ for $q\leq r$,
and $\tstA^s_q(\rho)\subset\tstA^\alpha_r(\rho)$ for $s>\alpha$ and for any $0<q,r\leq\infty$.
The set $\tstA^s_q(\rho)$ is not a linear space without further assumptions on $\rho$ and $\tstP$.
However, in a typical situation, it is indeed a vector space equipped with the quasi-norm $\|\cdot\|_{\tstA^s_q(\rho)} = \|\cdot\|_{X_0} + |\cdot|_{\tstA^s_q(\rho)}$.

\begin{remark}\label{r:quasi-norm}
Suppose that $\rho$ satisfies
\begin{itemize}
\item $\rho(\alpha u, \alpha v,P)=|\alpha|\rho(u, v,P)$ for $\alpha\in\R$, and
\item $\rho(u+u',v+v',P)\lesssim \rho(u,v,P) + \rho(u',v',P)$.
\end{itemize}
Then $|\cdot|_{\tstA^s_q(\rho)}$ is a {\em quasi-seminorm}, in the sense that it is positive homogeneous and satisfies the generalized triangle inequality.
Moreover, $\tstA^s_q(\rho)$ is a quasi-normed vector space.
If only the second condition holds, then $\tstA^s_q(\rho)$ would be a quasi-normed abelian group, in the sense of \cite{BL76}.
Even though we will not use this fact, it is worth noting that $\rho$ has the aforementioned properties for all applications we have in mind.
\end{remark}

We call the approximation classes associated to $\rho(u,v,\cdot)=\|u-v\|_{X_0}$ 
{\em standard approximation classes}, and write $\tstA^s_q(X_0)\equiv\tstA^s_q(\rho)$ and $\tstA^s(X_0)\equiv\tstA^s(\rho)$.
These standard spaces are within the scope of the general theory of approximation spaces, cf., \cite{Piet81,DL93}.
However, to treat the more general spaces $\tstA^s(\rho)$, the standard theory needs to be reworked, which is the aim of this section.

We want to characterize $\tstA^s(\rho)$ in terms of an auxiliary quasi-Banach space $X\hookrightarrow X_0$.
The main examples to keep in mind are $X_0=L^{p}$ and $X=B^{\alpha}_{q,q}$, with $\frac\alpha{n} > \frac1q - \frac1p$.
We assume that the space $X$ has the following local structure:
There exist a constant $0<q<\infty$, and a function $|\cdot|_{X(G)}:X\to\R^+$ associated to each open set $G\subset\Omega$, 
such that
\begin{equation}
\sum_{k} |u|_{X(\tau_k)}^q \lesssim \|u\|_{X}^q
\qquad (u\in X),
\end{equation}
for any finite sequence $\{\tau_k\}\subset P$ of non-overlapping elements taken from any $P\in\tstP$.
Finally, for any $\tau\in P$ with $P\in\tstP$, we let $\hat\tau\subset\Omega$ be a domain containing $\tau$,
which will, in a typical situation, be the union of elements of $P$ surrounding $\tau$.
We express the dependence of $\hat\tau$ on $P$ as $\hat\tau=P(\tau)$.
Then as an extension of the above sub-additivity property, we assume that 
\begin{equation}\label{e:loc-X-lower-overlap}
\sum_{k} |u|_{X(P_k(\tau_k))}^q \lesssim \|u\|_{X}^q
\qquad (u\in X),
\end{equation}
for any finite sequences $\{P_k\}\subset\tstP$ and $\{\tau_k\}$, with $\tau_k\in P_k$ and $\{\tau_k\}$ non-overlapping.
A trivial example of such a structure is $X=\Leb{q}(\Omega)$ with $|\cdot|_{X(G)}=\|\cdot\|_{\Leb{q}(G)}$.
Here the sub-additivity \eqref{e:loc-X-lower-overlap} can be guaranteed if the underlying triangulations satisfy a certain local finiteness property.

\subsection{Direct embeddings for standard approximation classes}
\label{ss:direct-standard}

The following theorem shows that the inclusion $X\subset\tstA^s(\rho)$ can be proved by exhibiting a direct estimate.
A direct application of this criterion is mainly useful for deriving embeddings of the form $X\subset\tstA^s(X_0)$.
In the next subsection, it will be generalized to a criterion that is valid in a more complex situations.

\begin{theorem}\label{t:direct-std}
Let $0< p\leq\infty$ and let $\delta>0$.
Assume \eqref{e:complete} on the complexity of completion, 
and assume \eqref{e:loc-X-lower-overlap} on the local structure of $X$.
Then for any $k\in\N$ sufficiently large there exists a partition $P\in\tstP$ with $\#P\leq k$ satisfying
\begin{equation}
\left( \sum_{\tau\in P} |\tau|^{p\delta} |u|_{X(\hat\tau)}^p \right)^{\frac1p} \lesssim k^{-s} \|u\|_{X} ,
\end{equation}
with $s=\delta + \frac1q - \frac1p>0$, 
where $\hat\tau=P(\tau)$ is as in \eqref{e:loc-X-lower-overlap}, 
and the case $p=\infty$ must be interpreted in the usual way (with a maximum replacing the discrete $p$-norm).
In particular, if $u\in X$ satisfies
\begin{equation}\label{e:rho-bound}
E(u,S_P)_\rho \lesssim \left( \sum_{\tau\in P} |\tau|^{p\delta} |u|_{X(\hat\tau)}^p \right)^{\frac1p} ,
\end{equation}
for all $P\in\tstP$,
then we have $u\in\tstA^s(\rho)$ with $|u|_{\tstA^s(\rho)}\lesssim \|u\|_X$.
\end{theorem}

\begin{proof}
What follows is a slight abstraction of the proof of Proposition 5.2 in \cite{BDDP02};
we include it here for completeness.
We first deal with the case $0<p<\infty$.
Let
\begin{equation}\label{e:err-ind-0}
e(\tau,P) = |\tau|^{p\delta} |u|_{X(\hat\tau)}^p,
\end{equation}
for $\tau\in P$ and $P\in\tstP$.
Then for any given $\eps>0$, and any $P_0\in\tstP$,
below we will specify a procedure to generate a partition $P\in\tstP$ satisfying 
\begin{equation}\label{e:rho-bnd-pf}
\sum_{\tau\in\tstP} e(\tau,P) \leq c' (\#P) \eps,
\end{equation}
and 
\begin{equation}\label{e:card-P-bnd-pf}
\#P - \#P_0 \leq c \eps^{-1/(1+ps)} \|u\|_{X}^{p/(1+ps)},
\end{equation}
where $c'$ depends only on the implicit constant of \eqref{e:rho-bound},
and $c$ depends only on $|\Omega|$, $\lambda$, and the implicit constants of \eqref{e:complete}, and \eqref{e:loc-X-lower-overlap}.
Then, for any given $k>0$, by choosing 
\begin{equation}
\eps = (c/k)^{1+ps} \|u\|_{X}^p,
\end{equation}
we would be able to guarantee a partition $P\in\tstP$ satisfying
$\#P\leq\#P_0+k$ and 
\begin{equation}
\sum_{\tau\in\tstP} e(\tau,P) \lesssim k^{-sp} \|u\|_{X}^p.
\end{equation}
This would imply the lemma, as $k^{-s}$ can be replaced by $(\#P_0+k)^{-s}$ for, e.g., $k\geq\#P_0$.

Let $\eps>0$ and let $P_0\in\tstP$.
We then recursively define $R_k=\{\tau\in P_k:e(\tau,P_k)>\eps\}$ and $P_{k+1}=\refine(P_k,R_k)$ for $k=0,1,\ldots$.
For all sufficiently large $k$ we will have $R_k=\varnothing$ since $|u|_{X(\hat\tau)} \lesssim \|u\|_{X}$ by \eqref{e:loc-X-lower-overlap},
and $|\tau|$ is reduced by a constant factor $\mu=\lambda^{-n}<1$ at each refinement.
Let $P=P_k$, where $k$ marks the first occurrence of $R_k=\varnothing$.
Then recalling \eqref{e:err-ind-0}, and taking into account that $e(\tau,P_k)\leq\eps$ for $\tau\in P_k$,
we obtain \eqref{e:rho-bnd-pf}.

In order to get a bound on $\#P$, we estimate the cardinality of $R=R_0\cup R_1\cup\ldots\cup R_{k-1}$, and use \eqref{e:complete}.
Let $\Lambda_j=\{\tau\in R: \mu^{j+1}\leq|\tau|<\mu^{j}\}$ for $j\in\Z$, and let $m_j=\#\Lambda_j$.
Note that the elements of $\Lambda_j$ (for any fixed $j$) are disjoint, since if any two elements intersect,
then they must come from different $R_k$'s as each $R_k$ consists of disjoint elements,
and hence by assumption on the refinement procedure, the ratio between the measures of the two elements must lie outside $(\mu,\mu^{-1})$.
This gives the trivial bound
\begin{equation}
m_j \leq \mu^{-j-1} |\Omega|.
\end{equation}
On the other hand, we have $e(\tau,P_k)>\eps$ for $\tau\in \Lambda_j$ with some $k$, which gives
\begin{equation}\label{e:bnd-epd-pf}
\eps < |\tau|^{p\delta} |u|_{X(\hat\tau)}^p
< \mu^{jp\delta} |u|_{X(\hat\tau)}^p,
\end{equation}
where $\hat\tau$ is defined with respect to $P_k$, and $k$ may depend on $\tau$.
Summing over $\tau\in\Lambda_j$, we get
\begin{equation}
m_j\eps^{q/p} 
\leq \mu^{jq\delta} \sum_{\tau\in \Lambda_j} |u|_{X(\hat\tau)}^q
\lesssim \mu^{jq\delta} \|u\|_{X}^q,
\end{equation}
where we have used \eqref{e:loc-X-lower-overlap}.
Finally, summing for $j$, we obtain
\begin{equation}
\#R
\leq \sum_{j=-\infty}^\infty m_j
\lesssim \sum_{j=-\infty}^\infty \min\left\{ \mu^{-j}, \eps^{-q/p} \mu^{jq\delta} \|u\|_{X}^q \right\}
\lesssim \eps^{-q/(p+pq\delta)} \|u\|_{X}^{q/(1+q\delta)},
\end{equation}
which, in view of \eqref{e:complete} and $q/(1+q\delta) = p/(1+ps)$, establishes the bound \eqref{e:card-P-bnd-pf}.

We only sketch the case $p=\infty$ since the proof is essentially the same.
We use
\begin{equation}
e(\tau,P) = |\tau|^{\delta} |u|_{X(\hat\tau)},
\end{equation}
instead of \eqref{e:err-ind-0}, and run the same algorithm.
This guarantees that the resulting partition $P$ satisfies
\begin{equation}
\max_{\tau\in P} e(\tau,P) \leq \eps.
\end{equation}
We bound the cardinality of $P$ in the same way, 
which formally amounts to putting $p=1$ into \eqref{e:bnd-epd-pf} and proceeding.
The final result is
\begin{equation}
\#P - \#P_0 \leq c \eps^{-q/(1+q\delta)} \|u\|_{X}^{q/(1+q\delta)} = c \eps^{-1/s} \|u\|_{X}^{1/s},
\end{equation}
where $s = \delta+\frac1q$.
The proof is complete.
\end{proof}

\begin{exampl}\label{eg:bdd-gm}
The main argument of the preceding proof can be traced back to \cite{BirSol67}.
Recently, in \cite{BDDP02}, this argument was applied to obtain an embedding of a Besov space into 
$\tstA^s(X_0)$, i.e., the case where the distance function $\rho$ is given by $\rho(u,v,\cdot)=\|u-v\|_{X_0}$.
We want to include here one such application. 
Let $\Omega\subset\R^n$ be a bounded polyhedral domain with Lipschitz boundary, 
and take $\tstP$ to be the set of conforming triangulations of $\Omega$ obtained from a fixed conforming triangulation $P_0$ by 
means of newest vertex bisections. 
For $P\in\tstP$, let $S_P$ be the Lagrange finite element space of continuous piecewise polynomials of degree not exceeding $m$.
Thus, for instance, the piecewise linear finite elements would correspond to $m=1$.
Moreover, for $P\in\tstP$ and $\tau\in P$, let $\hat\tau=P(\tau)$ be the interior of $\bigcup \{\overline\sigma:\sigma\in P,\,\overline\sigma\cap\overline\tau\neq\varnothing\}$.
Finally, let us put $X_0=\Leb{p}(\Omega)$, $X=B^{\alpha}_{q,q}(\Omega)$, and $\rho(u,v,\cdot)=\|u-v\|_{\Leb{p}(\Omega)}$.
Then in this setting, the estimate \eqref{e:rho-bound} holds with the parameters $p$ and $\delta=\frac\alpha{n}+\frac1p-\frac1q$,
as long as $0<\alpha<m+ \max\{1,\frac1q\}$ and $\delta>0$,
cf. \cite{BDDP02,GM13}.
Hence the preceding lemma immediately implies that $B^{\alpha}_{q,q}(\Omega)\hookrightarrow\tstA^s(\Leb{p}(\Omega))$ with $s=\frac\alpha{n}$.
\end{exampl}

In the rest of this subsection, we want to record some results involving interpolation spaces.
For $u\in X_0$ and $t>0$, the {\em $K$-functional} is
\begin{equation}\label{e:K-functional}
K(u,t;X_0,X) = \inf_{v\in X} \left( \|u-v\|_{X_0} + t \|v\|_X \right),
\end{equation}
and for $0<\theta<1$ and $0<\gamma\leq\infty$, we define the (real) {\em interpolation space}
$(X_0,X)_{\theta,\gamma}$ as the space of functions $u\in X_0$ for which the quantity
\begin{equation}\label{e:interp-norm-disc}
|u|_{(X_0,X)_{\theta,\gamma}} = \left\| [\lambda^{\theta m} K(u,\lambda^{-m};X_0,X)]_{m\geq0} \right\|_{\leb{\gamma}} ,
\end{equation}
is finite.
These are quasi-Banach spaces with the quasi-norms $\|\cdot\|_{X_0} + |\cdot|_{(X_0,X)_{\theta,\gamma}}$.
The parameter $\lambda>1$ can be chosen at one's convenience, because the resulting quasi-norms are all pairwise equivalent.

\begin{corollary}\label{c:direct-std}
Let $0< p\leq\infty$, $\delta>0$ and let $s=\delta + \frac1q - \frac1p>0$.
Assume
\begin{equation}
E(u,S_P)_{X_0} \lesssim \left( \sum_{\tau\in P} |\tau|^{p\delta} |u|_{X(\hat\tau)}^p \right)^{\frac1p} ,
\end{equation}
for all $u\in X$ and $P\in\tstP$.
Then we have $(X_0,X)_{\alpha/s,\gamma}\subset\tstA^\alpha_\gamma(X_0)$ for $0<\alpha<s$ and $0<\gamma\leq\infty$.
\end{corollary}

\begin{proof}
For $u\in (X_0,X)_{\alpha/s,\gamma}$ and $v\in X$, we have
\begin{equation}
E_k(u)_{X_0} \lesssim \|u-v\|_{X_0} + E_k(v)_{X_0} ,
\end{equation}
Theorem \ref{t:direct-std} yields $E_k(v)_{X_0}\lesssim 2^{-ks} \|v\|_X$, leading to
\begin{equation}
E_k(u)_{X_0} \lesssim \|u-v\|_{X_0} + 2^{-ks} \|v\|_X.
\end{equation}
After minimizing over $v\in X$,
the right hand side gives the $K$-functional $K(u,2^{-ks};X_0,X)$,
and \eqref{e:interp-norm-disc} implies that $u\in \tstA^\alpha_\gamma({X_0})$.
\end{proof}

\subsection{Direct embeddings for general approximation classes}
\label{ss:direct}

As mentioned in the introduction, our study is motivated by algorithms for approximating the solution of the operator equation $Tu=f$.
Hence it should not come as a surprise that we assume the existence of a continuous operator $T:X_0\to Y_0$, with $Y_0$ a quasi-Banach space.
An example to keep in mind is the Laplace operator sending $H^1_0$ onto $H^{-1}$.
In this subsection we do not assume linearity, although all examples of $T$ we have in this paper are linear.
We also need an auxiliary quasi-Banach space $Y\hookrightarrow Y_0$,
satisfying the properties analogous to that of $X$, in particular, \eqref{e:loc-X-lower-overlap} with some $0<r<\infty$ replacing $q$ there.
If $Y_0=H^{-1}$, a typical example of $Y$ would be $B^{\sigma-1}_{r,r}$ with $\frac\sigma{n} > \frac1r - \frac12$.

It is obvious that $\tstA^s(\rho)\subset\tstA^s(X_0)$, provided we have $\|u-v\|_{X_0}\lesssim\rho(u,v,\cdot)$.
The latter condition is satisfied for all practical applications we have in mind.
We will shortly present a theorem providing a criterion for affirming embeddings such as $\tstA^s(X_0)\cap T^{-1}(Y)\subset\tstA^s(\rho)$.

Before stating the theorem, we need to introduce a bit more structure on the set $\tstP$.
The structure we need is that of \emph{overlay} of partitions:
We assume that there is an operation
$\oplus:\tstP\times\tstP\to\tstP$ satisfying
\begin{equation}\label{e:overlay}
S_{P} + S_{Q} \subset S_{P\oplus Q},
\qquad
\textrm{and}
\qquad
\#(P\oplus Q)\lesssim \#P+\#Q,
\end{equation}
for $P,Q\in\tstP$.
In addition, we will assume that
\begin{equation}\label{e:overlay-rho}
\rho(u,v,P\oplus Q)\lesssim\rho(u,v,P).
\end{equation}
In the conforming world, $P\oplus Q$ can be taken to be the smallest
and common conforming refinement of $P$ and $Q$, for which
\eqref{e:overlay} is demonstrated in \cite{Stev07}, see also \cite{CKNS08}.
The same argument works for nonconforming partitions satisfying a certain admissibility criterion, cf. \cite{BN10}.

\begin{theorem}\label{t:direct}
Let $0< p\leq\infty$, $\delta>0$, and let $s=\delta + \frac1r - \frac1p>0$.
Assume \eqref{e:complete} on the complexity of completion, 
as well as \eqref{e:loc-X-lower-overlap} on the local structure of $Y$, with $r$ replacing $q$ there.
There are no additional assumptions at this point on $\rho$, except \eqref{e:overlay-rho} and the obvious conditions we have imposed in \S\ref{ss:setup}.
Suppose that $u\in\tstA^s(X_0)\cap T^{-1}(Y)$ satisfies
\begin{equation}\label{e:rho-bound-u-Au}
E(u,S_P)_\rho \lesssim E(u,S_P)_{X_0} 
+ \left( \sum_{\tau\in P} |\tau|^{p\delta} |Tu|_{Y(\hat\tau)}^p \right)^{\frac1p} ,
\end{equation}
for all $P\in\tstP$ (in particular $E(u,\cdot)_\rho$ is always finite).
Suppose also that 
\begin{equation}\label{e:Au-red}
\left( \sum_{\tau\in P\oplus Q} |\tau|^{p\delta} |Tu|_{Y(\hat\tau)}^p \right)^{\frac1p} \lesssim \left( \sum_{\tau\in P} |\tau|^{p\delta} |Tu|_{Y(\hat\tau)}^p \right)^{\frac1p} ,
\end{equation}
for any $P,Q\in\tstP$.
Then we have $u\in\tstA^s(\rho)$ with $|u|_{\tstA^s(\rho)}\lesssim |u|_{\tstA^s(X_0)}+\|Tu\|_Y$.
\end{theorem}

\begin{proof}
Let $k\in\N$ be an arbitrary number.
Then by definition of $\tstA^s(X_0)$, there exists a partition $P'\in\tstP$ such that
\begin{equation}
E(u,S_{P'})_{X_0} \leq 2^{-ks} |u|_{\tstA^s(X_0)},
\qquad\textrm{and}\qquad
\#P'\leq 2^kN .
\end{equation}
Similarly, by applying Theorem \ref{t:direct-std} with $Y$ in place of $X$, and with $Tu$ in place $u$,
we can generate a partition $P''\in\tstP$ such that
\begin{equation}
\left( \sum_{\tau\in P''} |\tau|^{p\delta} |Tu|_{Y(\hat\tau)}^p \right)^{\frac1p} \lesssim 2^{-ks} \|Tu\|_{Y},
\qquad\textrm{and}\qquad
\#P''\leq 2^kN .
\end{equation}
Then for $P=P'\oplus P''$ we have $\#P\lesssim2^k$ by \eqref{e:overlay}.
Moreover, \eqref{e:Au-red} together with the obvious monotonicity
\begin{equation}
E(u,S_P)_{X_0} \leq E(u,S_{P'})_{X_0},
\end{equation}
guarantee that the right hand side of \eqref{e:rho-bound-u-Au} is bounded by a multiple of $2^{-ks}(|u|_{\tstA^s(X_0)} + \|Tu\|_{Y})$,
which completes the proof.
\end{proof}

\begin{remark}
Suppose that we replace the condition \eqref{e:rho-bound-u-Au} in the statement of the preceding theorem by the new condition
\begin{equation}
E(u,S_P)_\rho \lesssim E(u,S_P)_{X_0} + E(u,S_P)_{\rho_1} 
+ \left( \sum_{\tau\in P} |\tau|^{p\delta} |Tu|_{Y(\hat\tau)}^p \right)^{\frac1p} ,
\end{equation}
where $\rho_1$ is some distance function.
Then by the same argument, we would be able to conclude that 
$u\in\tstA^s(\rho)$ with $|u|_{\tstA^s(\rho)}\lesssim |u|_{\tstA^s(X_0)}+|u|_{\tstA^s(\rho_1)}+\|Tu\|_Y$.
We will use similar strtaightforward extensions of the preceding theorem later in the paper, for instance, in the proof of Theorem \ref{t:elliptic-2nd}.
\end{remark}

\begin{exampl}\label{eg:Poisson}
We would like to illustrate the usefulness of Theorem \ref{t:direct} by sketching a simple application.
For full details, we refer to Section \ref{s:2nd-order}, as the current example is a special case of the results derived there.
We take $\Omega$ and $\tstP$ as in Example \ref{eg:bdd-gm},
and for $P\in\tstP$, let $S_P$ be the Lagrange finite element space of continuous piecewise polynomials of degree not exceeding $m$,
with the homogeneous Dirichlet boundary condition.
Moreover, we set $T=\Delta$ the Laplace operator, sending $X_0=H^1_0(\Omega)$ onto $Y_0=H^{-1}(\Omega)$.
Then in this context, it is proved in \cite{CKNS08} that certain adaptive finite element methods converge optimally
with respect to the approximation classes $\tstA^s(\rho)$, with the distance function $\rho$ given by
\begin{equation}\label{e:Poisson-rho-eg}
\rho(u,v,P)
= \left( \|u-v\|_{H^1}^2 
+ \sum_{\tau\in P} |\tau|^{2/n}\|f-\Pi_\tau f\|_{\Leb{2}(\tau)}^2 \right)^{\frac12} ,
\end{equation}
where $f=\Delta u$, and $\Pi_\tau:\Leb{2}(\tau)\to\Pol_{d}$ is the $\Leb{2}(\tau)$-orthogonal projection onto $\Pol_{d}$,
with $d\geq m-2$ fixed.
The sum involving $f-\Pi_\tau f$ is known as the {\em oscillation term}.

Let $0< r,\alpha<\infty$ satisfy $\delta = \frac{\alpha}n - \frac1r + \frac12\geq0$ and $\alpha<d+\max\{1,\frac1r\}$.
Then we claim that for each $u\in H^1_0(\Omega)$ with $\Delta u\in B^{\alpha}_{r,r}(\Omega)$, 
there exists $u_P\in S_P$ such that
\begin{equation}
\rho(u,u_P,P) \lesssim \inf_{v\in S_P} \|u-v\|_{H^1} + \left( \sum_{\tau\in P} |\tau|^{2(\delta+1/n)} |\Delta u|_{B^{\alpha}_{r,r}(\tau)}^2 \right)^{\frac12},
\end{equation}
for all $P\in\tstP$.
In light of the preceding theorem, this would imply that each function $u\in\tstA^s(H^1_0(\Omega))$ with $\Delta u\in B^{\alpha}_{r,r}(\Omega)$,
satisfies $u\in\tstA^s(\rho)$, cf. Figure \ref{f:direct-intro}(b).
Note that since we can choose $d$ at will, the restriction $\alpha<d+\max\{1,\frac1r\}$ is immaterial.

To prove the claim, we take $u_P$ to be the Scott-Zhang interpolator of $u$, preserving the Dirichlet boundary condition.
Then we have
\begin{equation}
\|u-u_P\|_{H^1} \lesssim \inf_{v\in S_P} \|u-v\|_{H^1},
\end{equation}
for all $P\in\tstP$.
The oscillation term in \eqref{e:Poisson-rho-eg} can be estimated as
\begin{equation}
\|f-\Pi_\tau f\|_{\Leb{2}(\tau)} 
\leq \|f - g\|_{\Leb{2}(\tau)} 
\lesssim 
|\tau|^{\delta} \|f-g\|_{\Leb{r}(\tau)} + |\tau|^{\delta} |f|_{B^{\alpha}_{r,r}(\tau)},
\end{equation}
for any $g\in\Pol_{d}$,
where we have used continuity of the embedding $B^{\alpha}_{r,r}(\tau)\subset \Leb{2}(\tau)$
and the fact that $|g|_{B^{\alpha}_{r,r}(\tau)}=0$ when the Besov seminorm is defined using $\omega_{d+1}$.
Furthermore, if $g$ is a best approximation of $f$ in the $\Leb{r}(\tau)$ sense,
the Whitney estimate gives
\begin{equation}
\|f-g\|_{\Leb{r}(\tau)} 
\lesssim \omega_{d+1}(f,\tau)_r
\lesssim |f|_{B^{\alpha}_{r,r}(\tau)},
\end{equation}
which yields the desired result.
In closing the example, we note that for this argument to work, 
the constants in the Whitney estimates and in the embeddings $B^{\alpha}_{r,r}(\tau)\subset \Leb{2}(\tau)$ must be uniformly bounded independently of $\tau$.
While such investigations on Whitney estimates can be found in \cite{DL04a,GM13}, it seems difficult to locate similar studies on Besov space embeddings.
To remove any doubt, the arguments in the following sections are arranged so that we do not use Besov space embeddings.
Instead, we use embeddings between approximation spaces, and give a self contained proof that the embedding constants are suitably controlled.
\end{exampl}


\section{Lagrange finite elements}
\label{s:Lagrange}

\subsection{Preliminaries}

Let $\Omega\subset\R^n$ be a bounded domain.
Then for $0<p\leq\infty$, we define the $r$-th order {\em $\Leb{p}$-modulus of smoothness}
\begin{equation}
\omega_r(u,t,\Omega)_p=\sup_{|h|\leq t}\|\Delta_h^ru\|_{\Leb{p}(\Omega_{rh})}
\end{equation}
where $\Omega_{rh}=\{x\in\Omega:[x,x+rh]\subset\Omega\}$,
and $\Delta_h^r$ is the $r$-th order forward difference operator defined recursively by $[\Delta_h^1u](x)=u(x+h)-u(x)$ and $\Delta_h^ku=\Delta_h^1(\Delta_h^{k-1})u$,
i.e.,
\begin{equation}
\Delta_h^ru (x) = \sum_{k=0}^r (-1)^{r+k} \binom{r}{k} u(x+kh).
\end{equation}
Furthermore, for $0<p,q\leq\infty$, $\alpha\geq0$, and $r\in\N$,
the {\em Besov space} $B^\alpha_{p,q;r}(\Omega)$ consists of those $u\in \Leb{p}(\Omega)$ for which
\begin{equation}
|u|_{B^\alpha_{p,q;r}(\Omega)} = \| t\mapsto t^{-\alpha-1/q}\omega_r(u,t,\Omega)_p \|_{\Leb{q}((0,\infty))},
\end{equation}
is finite.
Since $\Omega$ is bounded, being in a Besov space is a statement about the size of $\omega_r(u,t,\Omega)_p$ only for small $t$.
From this it is easy to derive the useful equivalence
\begin{equation}\label{e:Besov-norm-disc}
|u|_{B^\alpha_{p,q;r}(\Omega)} \eqsim \left\| (\lambda^{j\alpha}\omega_r(u,\lambda^{-j},\Omega)_p)_{j\geq0} \right\|_{\leb{q}},
\end{equation}
for any constant $\lambda>1$.
The mapping $\|\cdot\|_{B^\alpha_{p,q;r}(\Omega)}=\|\cdot\|_{\Leb{p}(\Omega)}+|\cdot|_{B^\alpha_{p,q;r}(\Omega)}$ defines a norm when $p,q\geq1$ and a quasi-norm in general.
If $\alpha>r+\max\{0,\frac1p-1\}$ then the space $B^\alpha_{p,q;r}$ is trivial in the sense that $B^\alpha_{p,q;r}=\Pol_{r-1}$.
On the other hand, so long as $r>\alpha-\max\{0,\frac1p-1\}$, different choices of $r$ will result in quasi-norms that are equivalent to each other, and in this case we have the classical Besov spaces $B^\alpha_{p,q}(\Omega)=B^\alpha_{p,q;r}(\Omega)$.
In the borderline case, the situation depends on the index $q$.
If $0<q<\infty$ and $\alpha=r+\max\{0,\frac1p-1\}$, then $B^\alpha_{p,q;r}=\Pol_{r-1}$.
The case $q=\infty$ gives nontrivial spaces: For instance, we have $B^r_{p,\infty;r}(\Omega)=W^{r,p}(\Omega)$ for $p>1$.
A proof can be found in \cite[page 53]{DL93} for the one dimensional case, and the same proof works in multi-dimensions.

The following result, often called the {\em discrete Hardy inequality}, will be used many times in the subsequent sections.
We include the statement here for convenience. A proof can be found in \cite[page 27]{DL93}.

\begin{lemma}\label{e:Hardy-ineq-disc}
Let $(a_j)_{j\in\Z}$ and $(b_k)_{k\in\Z}$ be two sequences satisfying either
\begin{equation}\label{e:Hardy-ineq-disc-1}
|a_j| \leq C \left( \sum_{k=j}^\infty |b_k|^\mu \right)^{1/\mu}, \qquad j\in\Z,
\end{equation}
for some $\mu>0$ and $C>0$, or
\begin{equation}\label{e:Hardy-ineq-disc-2}
|a_j| \leq C 2^{-\theta j} \left( \sum_{k=-\infty}^j |2^{\theta k} b_k|^\mu \right)^{1/\mu}, \qquad j\in\Z.
\end{equation}
for some positive $\theta$, $\mu$, and $C$.
Then we have
\begin{equation}\label{e:Hardy-ineq-disc-3}
\|(2^{\alpha j}a_j)_j\|_{\leb{q}} \lesssim C\|(2^{\alpha k}b_k)_k\|_{\leb{q}},
\end{equation}
for all $0<q\leq\infty$ and $0<\alpha<\theta$, with the convention that $\theta=\infty$ if \eqref{e:Hardy-ineq-disc-1} holds,
and with the implicit constant depending only on $q$ and $\alpha$.
\end{lemma}

\subsection{Quasi-interpolation operators}

Let $\Omega\subset\R^n$ be a bounded polyhedral domain with Lipschitz boundary, 
and fix a conforming partition $P_0$ of $\Omega$.
We also fix a refinement rule, which is either the newest vertex bisection or the red refinement.
This ensures that all partitions are shape regular.
The set $\tstP$ can be chosen to be, in case of the newest vertex bisection, the set of all conforming triangulations arising from $P_0$.
More generally, we want to deal with possibly nonconforming partitions,
and we require that $\tstP$ satisfy the {\em finite support condition} \eqref{e:finite-support} stated below.
We remark that the results in this paper are insensitive to the exact definition of $\tstP$, so long as the family $\tstP$ satisfies \eqref{e:finite-support}.
We note that for nonconforming partitions, the degrees of freedom will be so arranged that they give rise to $H^1$-conforming finite element spaces.
In this regard, the terminology ``nonconforming partition'' may be a bit confusing.

We define the Lagrange finite element spaces
\begin{equation}\label{e:Lagrange-fem-space}
{S}_{P}={S}^m_{P}=\left\{u\in C(\Omega) : u|_{\tau}\in \mathbb{P}_{m}\,\forall\tau\in P\right\} ,
\qquad P\in\tstP ,
\end{equation}
where $\Pol_m$ is the space of polynomials of degree not exceeding $m$.
Thus, for instance, the piecewise linear finite elements would correspond to $m=1$.

Following \cite{GM13}, we will now construct a quasi-interpolation operator $\tilde Q_{P}:\Leb{p_0}(\Omega)\to S_P$ for $p_0>0$ small.
Their construction works almost verbatim here but we need to be a bit careful since we want to include partitions with hanging nodes into the analysis.
Let $\tau_0=\{x\in\R^n:x_1+\ldots+x_n<m\}\cap(0,m)^n$ be the standard simplex.
Then an $n$-simplex $\tau\subset\R^n$ is the image of $\tau_0$ under an invertible affine mapping.
To each  $n$-simplex $\tau$, we associate its nodal set $N_\tau=F(\bar\tau_0\cap\Z^n)$, where $F:\tau_0\to\tau$ is any invertible affine mapping.
The {\em nodal set} $N_P$ of a possibly nonconforming partition $P\in\tstP$ is defined by the requirement that 
$z\in\bigcup_{\tau\in P} N_\tau$ is in  $N_P$ if and only if $z\in N_\tau$ for all $\tau$ satisfying $\bar\tau\ni z$, see Figure \ref{f:Lagrange-nodes}.
Furthermore, we define the {\em nodal basis} $\{\phi_z:z\in N_P\}\subset S_P$ of $S_P$ by $\phi_z(z')=\delta_{z,z'}$ for $z,z'\in N_P$.

The aforementioned {\em finite support condition} is as follows.
We require that there is a constant $C>0$ independent of $P\in\tstP$ and $z\in N_P$, such that 
\begin{equation}\label{e:finite-support}
\#\{\tau\in P:\tau\subset\supp\,\phi_z\} \leq C,
\end{equation}
for all $P\in\tstP$ and $z\in N_P$.
It is obvious that conforming triangulations satisfy this requirement.
For partitions with hanging nodes, we refer to \cite{BN10}.

Given $\tau\in P$ and $\sigma\in P$, let us write $\tau\sim\sigma$ if (and only if) there is $z\in N_P$ such that $\tau\cup\sigma\subset\supp\,\psi_z$.
Then, taking into account the refinement rule and the definition of the nodal basis, 
one can show that the finite support condition is equivalent to the {\em strong local finiteness condition}:
\begin{equation}\label{e:local-finite}
\sup_{P\in\tstP}\sup_{\tau\in P}\#\{\sigma\in P:\sigma\sim\tau\} <\infty ,
\end{equation}
which in turn is equivalent to the {\em strong gradedness condition}:
\begin{equation}\label{e:graded}
\sup_{P\in\tstP}\big\{\frac{\diam\,\sigma}{\diam\,\tau}:\tau,\sigma\in P,\,\sigma\sim\tau\big\} <\infty .
\end{equation}
By defining the {\em support extension} $\hat\tau=P(\tau)$ of $\tau\in P$ as the interior of 
\begin{equation}\label{e:supp-ext}
\bigcup \{\sigma\in P:\sigma\sim\tau\} ,
\end{equation}
we can also write the strong gradedness condition as
\begin{equation}
\sup_{P\in\tstP}\sup_{\tau\in P} \frac{\diam\,\hat\tau}{\diam\,\tau} < \infty .
\end{equation}
In what follows, the implicit constant in any of the aforementioned conditions will be referred to as an {\em admissibility constant}.

\begin{figure}[ht]
\centering
\begin{subfigure}{0.35\textwidth}
\includegraphics{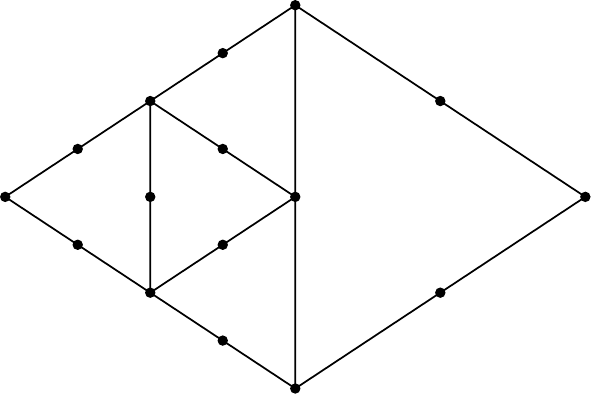}
\subcaption{Quadratic elements ($m=2$).}
\end{subfigure}
\qquad\qquad\qquad
\begin{subfigure}{0.35\textwidth}
\includegraphics{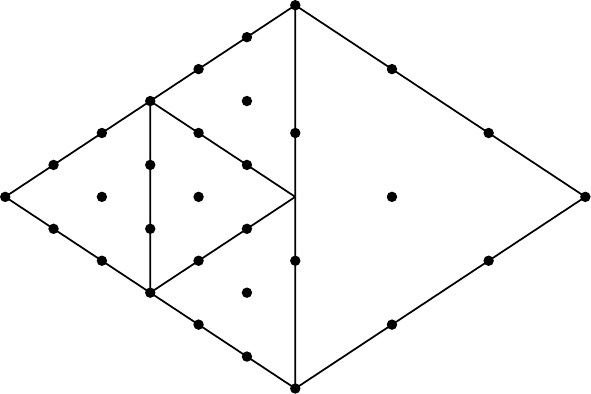}
\subcaption{Cubic elements ($m=3$).}
\end{subfigure}
\caption{Examples of nodal sets.}
\label{f:Lagrange-nodes}
\end{figure}

Next, we introduce a basis dual to $\{\phi_z\}$.
For each $\tau\in P$, we let
\begin{equation}
N_{P,\tau}=\{z\in N_P:\tau\subset\supp\,\phi_z\} ,
\end{equation}
and define $\eta_{\tau,z}\in\Pol_m$, $z\in N_{P,\tau}$,  by the condition 
\begin{equation}
\int_\tau \eta_{\tau,z} \xi_{\tau,z'} = \delta_{z,z'},
\qquad  z,z'\in N_{P,\tau} .
\end{equation}
Note that $\#N_{P,\tau}=\#N_\tau=\dim\Pol_m$, so that the set $\{\eta_{\tau,z}:z\in N_{P,\tau}\}$ is uniquely determined.
Then for $z\in N_P$, we let
\begin{equation}
\tilde\phi_z = \frac1{n_z} \sum_{\{\tau\in P:\,\tau\subset\supp\,\phi_z\}} \chi_\tau\eta_{\tau,z} ,
\end{equation}
where $n_z=\#\{\tau\in P:\tau\subset\supp\,\phi_z\}$, and $\chi_\tau$ is the characteristic function of $\tau$.
By construction, $\supp\,\tilde\phi_z=\supp\,\phi_z$ for $z\in N_P$, and we have the biorthogonality
\begin{equation}
\langle\tilde\phi_z,\phi_{z'}\rangle = \int_\Omega \tilde\phi_z \phi_{z'} = \delta_{z,z'},
\qquad  z,z'\in N_P .
\end{equation}
Now we define the quasi-interpolation operator $Q_P:\Leb{1}(\Omega)\to S_P$ by
\begin{equation}\label{e:quasi-interpolator-std}
Q_Pu = Q_{P}^{(\Omega)}u = \sum_{z\in N_P} \langle u,\tilde\phi_z\rangle\phi_z .
\end{equation}
It is clear that $Q_P$ is linear and that $Q_Pv=v$ for $v\in S_P$.

\begin{lemma}\label{l:quasi-interp-std}
For $1\leq p\leq\infty$, we have
\begin{equation}\label{e:quasi-interp-global-best-std}
\|u - Q_{P} u\|_{\Leb{p}(\Omega)} \lesssim \inf_{v\in S_P} \|u-v\|_{\Leb{p}(\Omega)} ,
\qquad u\in \Leb{p}(\Omega) ,
\end{equation}
with the implicit constant depending only on the shape regularity and admissibility constants of $\tstP$.
Furthermore, for $0<p\leq\infty$ and $\tau\in P$, we have
\begin{equation}\label{e:quasi-interp-bdd-disc-std}
\|Q_{P} v\|_{\Leb{p}(\tau)} \lesssim \|v\|_{\Leb{p}(\hat\tau)} ,
\qquad v\in \bar S^m_P ,
\end{equation}
where 
\begin{equation}\label{e:disc-pol-1}
\bar S^m_P = \{w\in \Leb{\infty}(\Omega):w|_\tau\in\Pol_m\,\forall\tau\in P\} ,
\end{equation}
and $\hat\tau=P(\tau)$ is the support extension of $\tau$ as defined in \eqref{e:supp-ext}.
\end{lemma}

\begin{proof}
For $1\leq p\leq\infty$ and $u\in \Leb{p}(\Omega)$, we have
\begin{equation}\label{e:quasi-interp-bdd-std-pf}
\|Q_Pu\|_{\Leb{p}(\tau)} 
\leq \sum_{z\in N_{P,\tau}} |\langle u,\tilde\phi_z\rangle| \, \|\phi_z\|_{\Leb{p}} 
\leq \|u\|_{\Leb{p}(\hat\tau)} \sum_{z\in N_{P,\tau}} \|\tilde\phi_z\|_{\Leb{q}} \|\phi_z\|_{\Leb{p}} ,
\end{equation}
where $\frac1p+\frac1q=1$.
It is clear that $\|\phi_z\|_{\Leb{p}}\leq|\supp\,\phi_z|^{1/p}$ and by a scaling argument one can deduce that
$\|\tilde\phi_z\|_{\Leb{q}}\lesssim|\supp\,\tilde\phi_z|^{1/q-1}$, with the implicit constant depending only on the shape regularity and admissibility constants of $\tstP$.
Consequently, for $1\leq p<\infty$, we infer
\begin{equation}
\|Q_Pu\|_{\Leb{p}(\Omega)}^p 
= \sum_{\tau\in P} \|Q_Pu\|_{\Leb{p}(\tau)}^p
\lesssim \sum_{\tau\in P} \|u\|_{\Leb{p}(\hat\tau)}^p
\lesssim \|u\|_{\Leb{p}(\Omega)}^p ,
\end{equation}
by the strong local finiteness of the mesh.
The case $p=\infty$ can be handled similarly, and we have $\|Q_Pu\|_{\Leb{p}(\Omega)}\lesssim\|u\|_{\Leb{p}(\Omega)}$ for $1\leq p\leq\infty$.
Then a standard argument yields \eqref{e:quasi-interp-global-best-std}.

For $1\leq p\leq\infty$, \eqref{e:quasi-interp-bdd-std-pf} implies \eqref{e:quasi-interp-bdd-disc-std}.
The proof for $0<p<1$ follows exactly the same lines as those in the proof of \cite[Lemma 3.2]{GM13}.
\end{proof}

In the following, we fix $p_0>0$, and for $\tau\subset\R^n$ a domain, 
let $\Pi_{p_0,\tau}:\Leb{p_0}(\tau)\to\Pol_m$ be the local polynomial approximation operator given in Definition 3.7 of \cite{GM13}.
We recall the following important properties of this operator, cf. \cite[Theorem 3.8]{GM13}.
\begin{enumerate}[(i)]
\item
There is a constant $C_{m,p_0}$ depending only on $m$ and $p_0$, such that
\begin{equation}\label{e:Pi-near-best}
\|u-\Pi_{p_0,\tau} u\|_{\Leb{p_0}(\tau)} \leq C_{m,p_0} \inf_{v\in\Pol_m} \|u-v\|_{\Leb{p_0}(\tau)} ,
\qquad u\in \Leb{p_0}(\tau) .
\end{equation}
In other words, $\Pi_{p_0,\tau} u$ is a near-best approximation of $u$ from $\Pol_m$ in $\Leb{p_0}(\tau)$.
\item
We have
\begin{equation}\label{e:Pi-bdd}
\|\Pi_{p_0,\tau} u\|_{\Leb{p_0}(\tau)} \lesssim \|u\|_{\Leb{p_0}(\tau)} ,
\qquad u\in \Leb{p_0}(\tau) ,
\end{equation}
i.e., the operator $\Pi_{p_0,\tau}:\Leb{p_0}(\tau)\to \Leb{p_0}(\tau)$ is bounded.
\item
For any $u\in \Leb{p_0}(\tau)$ and $v\in \Pol_m$, we have
\begin{equation}\label{e:Pi-linear}
\Pi_{p_0,\tau} (u+v) = \Pi_{p_0,\tau} u+v .
\end{equation}
In particular, $\Pi_{p_0,\tau}v=v$ for $v\in \Pol_m$.
\end{enumerate}

Finally, we let
\begin{equation}\label{e:quasi-interpolator-disc}
\Pi_Pu = \sum_{\tau\in P} \chi_\tau \Pi_{p_0,\tau} u ,
\end{equation}
and define the operator $\tilde Q_{P}:\Leb{p_0}(\Omega)\to S_P$ by
\begin{equation}\label{e:quasi-interpolator-Lagrange}
\tilde Q_Pu = Q_P\Pi_Pu = \sum_{z\in N_P} \langle \Pi_Pu,\tilde\phi_z\rangle\phi_z .
\end{equation}
It is easy to see that $\tilde Q_Pv=v$ for $v\in S_P$, and that $(\tilde Q_Pu)|_\tau$ depends only on $u|_{\hat\tau}$,
where $\hat\tau=P(\tau)$ is the support extension of $\tau$, as defined in \eqref{e:supp-ext}.
Furthermore, as a consequence of the linearity property \eqref{e:Pi-linear}, we have
\begin{equation}\label{e:quasi-interpolator-linearity}
( \tilde Q_P(u+v) ) |_{\tau} = ( \tilde Q_Pu ) |_{\tau} + v|_\tau ,
\qquad u,v\in \Leb{p_0}(\tau) , \quad v|_{\hat\tau} \in S_P ,
\end{equation}
for $\tau\in P$.

\begin{lemma}\label{l:quasi-interp-best-Lagrange}
Let $p_0\leq p\leq\infty$ and $P\in\tstP$.
Then for $\tau\in P$ we have
\begin{equation}\label{e:quasi-interp-stable}
\|\tilde Q_{P} u\|_{\Leb{p}(\tau)} \lesssim \|u\|_{\Leb{p}(\hat\tau)} ,
\qquad u\in \Leb{p}(\Omega) .
\end{equation}
As a consequence, we have
\begin{equation}\label{e:quasi-interp-local-best-Lagrange}
\|u - \tilde Q_{P} u\|_{\Leb{p}(\tau)} \lesssim \inf_{v\in S_P} \|u-v\|_{\Leb{p}(\hat\tau)} ,
\qquad u\in \Leb{p}(\Omega) ,
\end{equation}
and
\begin{equation}\label{e:quasi-interp-global-best-Lagrange}
\|u - \tilde Q_{P} u\|_{\Leb{p}(\Omega)} \lesssim \inf_{v\in S_P} \|u-v\|_{\Leb{p}(\Omega)} ,
\qquad u\in \Leb{p}(\Omega) .
\end{equation}
\end{lemma}

\begin{proof}
An application of \eqref{e:quasi-interp-bdd-disc-std} gives
\begin{equation}
\|\tilde Q_{P} u\|_{\Leb{p}(\tau)} = \|Q_{P} \Pi_P u\|_{\Leb{p}(\tau)} \lesssim \|\Pi_P u\|_{\Leb{p}(\hat\tau)} .
\end{equation}
On the other hand, for $\sigma\in P$, we have
\begin{equation}
\|\Pi_{p_0,\sigma} u\|_{\Leb{p}(\sigma)} 
\lesssim |\sigma|^{\frac1p-\frac1{p_0}} \|\Pi_{p_0,\sigma} u\|_{\Leb{p_0}(\sigma)}
\lesssim |\sigma|^{\frac1p-\frac1{p_0}} \|u\|_{\Leb{p_0}(\sigma)}
\leq \|u\|_{\Leb{p}(\sigma)} ,
\end{equation}
where we have used scaling properties of polynomials in the first step,
the boundedness \eqref{e:Pi-bdd} of $\Pi_{p_0,\sigma}$ in the second step,
and the H\"older inequality in the final step.
Using this, with the usual modifications for $p=\infty$, we infer
\begin{equation}
\|\Pi_P u\|_{\Leb{p}(\hat\tau)}
= \left( \sum_{\{\sigma\in P:\sigma\subset\hat\tau\}}\|\Pi_{p_0,\sigma} u\|_{\Leb{p}(\sigma)}^p \right)^{\frac1p}
\lesssim \left( \sum_{\{\sigma\in P:\sigma\subset\hat\tau\}}\|u\|_{\Leb{p}(\sigma)}^p \right)^{\frac1p}
= \|u\|_{\Leb{p}(\hat\tau)} ,
\end{equation}
establishing \eqref{e:quasi-interp-stable}.
Then \eqref{e:quasi-interp-local-best-Lagrange} follows from the linearity property \eqref{e:quasi-interpolator-linearity}.

The estimate \eqref{e:quasi-interp-global-best-Lagrange} is proved by first deriving the stability 
\begin{equation}
\|\tilde Q_{P} u\|_{\Leb{p}(\Omega)} 
= \left( \sum_{\tau\in P} \|\tilde Q_{P} u\|_{\Leb{p}(\tau)}^p \right)^{\frac1p} 
\lesssim \left( \sum_{\tau\in P} \|u\|_{\Leb{p}(\hat\tau)}^p \right)^{\frac1p} 
\lesssim \|u\|_{\Leb{p}(\Omega)} ,
\end{equation}
and then invoking the linearity property \eqref{e:quasi-interpolator-linearity}.
\end{proof}

An important tool in approximation theory is the {\em Whitney estimate}
\begin{equation}\label{e:Whitney-est}
\inf_{v\in\Pol_m} \|u-v\|_{\Leb{p}(G)} \lesssim \omega_{m+1}(u,\diam\,G,G)_p,
\qquad u\in \Leb{p}(G),
\end{equation}
that holds for any convex domain $G\subset\R^n$, 
with the implicit constant depending only on $n$, $m$, and $0<p\leq\infty$, see \cite{DL04a}.
The same estimate is also true when $G$ is the star around $\tau\in P$ for some partition $P\in\tstP$,
with the implicit constant additionally depending on the shape regularity constant of $\tstP$, see \cite{GM13}.


\begin{lemma}\label{l:quasi-interp-Lagrange}
Let $p_0\leq p\leq\infty$ and let $P\in\tstP$ be {\em conforming}.
Then we have
\begin{equation}\label{e:Jackson-Qj}
\|u - \tilde Q_{P} u\|_{\Leb{p}(\Omega)} 
\lesssim \left( \frac{\max_{\tau\in P}\diam\,\tau}{\min_{\tau\in P}\diam\,\tau} \right)^{\frac{n}p} \omega_{m+1}(u,\max_{\tau\in P}\diam\,\tau,\Omega)_p ,
\qquad u\in \Leb{p}(\Omega) .
\end{equation}
\end{lemma}

\begin{proof}
We start with the special case $p=\infty$.
Note that since $P$ is conforming, the support extension $\hat\tau$ of $\tau$ coincides with the star around $\tau$.
It is immediate from \eqref{e:quasi-interp-local-best-Lagrange} and the Whitney estimate \eqref{e:Whitney-est} that
\begin{equation}
\begin{split}
\|u - \tilde Q_{P} u\|_{\Leb{\infty}(\Omega)}
&\lesssim \max_{\tau\in P} \|u - \tilde Q_{P} u\|_{\Leb{\infty}(\tau)} 
\lesssim \max_{\tau\in P} \inf_{v\in S_P} \|u-v\|_{\Leb{\infty}(\hat\tau)} \\
&\lesssim \max_{\tau\in P} \omega_{m+1}(u,\diam\,\hat\tau,\hat\tau)_\infty
\lesssim \max_{\tau\in P} \omega_{m+1}(u,\mu^{-1}\diam\,\hat\tau,\hat\tau)_\infty ,
\end{split}
\end{equation}
with $\mu>0$ sufficiently large,
where in the last line we have used the property
\begin{equation}
\omega_r(u,\mu t,G)_p \leq (\mu+1)^r \omega_r(u,t,G)_p ,
\end{equation}
cf. \cite[\S2.7]{DL93}.
With $t=\displaystyle\mu^{-1}\max_{\tau\in P}\diam\,\hat\tau$, we proceed as
\begin{equation}
\begin{split}
\|u - \tilde Q_{P} u\|_{\Leb{\infty}(\Omega)}
&\lesssim \max_{\tau\in P} \omega_{m+1}(u,t,\hat\tau)_\infty
= \max_{\tau\in P} \sup_{|h|\leq t} \|\Delta^{m+1}_hu\|_{\Leb{\infty}(\hat\tau_{rh})} \\
&= \sup_{|h|\leq t} \max_{\tau\in P} \|\Delta^{m+1}_hu\|_{\Leb{\infty}(\hat\tau_{rh})} 
\leq \sup_{|h|\leq t} \|\Delta^{m+1}_hu\|_{\Leb{\infty}(\Omega_{rh})} ,
\end{split}
\end{equation}
which establishes \eqref{e:Jackson-Qj} for $p=\infty$.

To handle the case $0<p<\infty$ we introduce the averaged $\Leb{p}$-modulus of smoothness
\begin{equation}
w_r(u,t,G)_p = \left( \frac1{t^n} \int_{[0,t]^n} \|\Delta_h^ru\|_{\Leb{p}(G_{rh})}^p \exd h \right)^{1/p} ,
\end{equation}
for any domain $G\subset\R^n$.
When $G$ is Lipschitz, the averaged modulus is equivalent to the original one:
\begin{equation}\label{e:avg-modul-equiv}
w_r(u,t,G)_p \sim \omega_r(u,t,G)_p,
\qquad \textrm{for}\quad t\lesssim1 .
\end{equation}
This equivalence is also true when $G=\tau$ or $G=\hat\tau$ for $\tau\in P$ with $P\in\tstP$,
in the range $t\lesssim\diam\,G$, cf. Corollary 4.3 of \cite{GM13}.
In the latter case, 
the implicit constants depend only on $p$, $r$, the shape regularity constant of $\tstP$, and the geometry of the underlying domain $\Omega$.

Let us get back to the proof of \eqref{e:Jackson-Qj} for $0<p<\infty$.
As in the case $p=\infty$, we have
\begin{equation}
\begin{split}
\|u - \tilde Q_{P} u\|_{\Leb{p}(\Omega)}^p 
&\lesssim \sum_{\tau\in P} \|u - \tilde Q_{P} u\|_{\Leb{p}(\tau)}^p 
\lesssim \sum_{\tau\in P} \inf_{v\in S_P} \|u-v\|_{\Leb{p}(\hat\tau)}^p \\
&\lesssim \sum_{\tau\in P} \omega_{m+1}(u,\diam\,\hat\tau,\hat\tau)_p^p
\lesssim \sum_{\tau\in P} \omega_{m+1}(u,\mu^{-1}\diam\,\hat\tau,\hat\tau)_p^p ,
\end{split}
\end{equation}
with $\mu>0$ sufficiently large.
Now we employ \eqref{e:avg-modul-equiv}, to get
\begin{equation}
\begin{split}
\|u - \tilde Q_{P} u\|_{\Leb{p}(\Omega)}^p 
&\lesssim \sum_{\tau\in P} w_{m+1}(u,\mu^{-1}\diam\,\hat\tau,\hat\tau)_p^p \\
&= \sum_{\tau\in P} \frac1{t(\tau)^n} \int_{[0,t(\tau)]^n} \int_{\hat\tau_{rh}} |\Delta_h^ru(x)|^p \exd x \, \exd h ,
\end{split}
\end{equation}
where $t(\tau)=\mu^{-1}\diam\,\hat\tau$ and $r=m+1$.
With $t_0=\mu^{-1}\displaystyle\min_{\tau\in P}\diam\,\hat\tau$ and $t_1=\mu^{-1}\displaystyle\max_{\tau\in P}\diam\,\hat\tau$,
we can switch the sum with the outer integration as follows.
\begin{equation}
\begin{split}
\|u - \tilde Q_{P} u\|_{\Leb{p}(\Omega)}^p 
&\lesssim \frac1{t_0^n} \int_{[0,t_1]^n} \sum_{\tau\in P} \int_{\hat\tau_{rh}} |\Delta_h^ru(x)|^p \exd x \, \exd h \\
&\lesssim \frac1{t_0^n} \int_{[0,t_1]^n} \int_{\Omega_{rh}} |\Delta_h^ru(x)|^p \exd x \, \exd h \\
&= \frac{t_1^n}{t_0^n} w_r(u,t_1,\Omega)_p^p .
\end{split}
\end{equation}
The proof is completed upon using the equivalence \eqref{e:avg-modul-equiv} for $G=\Omega$.
\end{proof}

\subsection{Multilevel approximation spaces}
\label{ss:multilevel}

In this subsection, we study approximation from uniformly refined Lagrange finite element spaces.
We keep the setting of the preceding subsection intact,
and define the partitions $P_j$ for $j=1,2,\ldots$ recursively as $P_{j+1}$ is the uniform refinement of $P_j$.
Let $G\subset\Omega$ be a domain consisting of elements from some $P_j$.
More precisely, let $G$ be the interior of $\bigcup_{\tau\in Q}\bar\tau$ for some $Q\subset P_j$ and $j$.
Then with $S_j=S_{P_j}$, and $0<p\leq\infty$, we let
\begin{equation}\label{e:Lp-best}
E(u,S_j)_{\Leb{p}(G)} = \inf_{v\in S_j} \|u-v\|_{\Leb{p}(G)} ,
\qquad u\in \Leb{p}(G) .
\end{equation}
Note that the infimum is achieved since $S_j$ is a finite dimensional space.
We define the {\em multilevel approximation spaces} 
\begin{equation}\label{e:multilevel-approx}
A^\alpha_{p,q}(\{S_j\},G) = \left\{u\in \Leb{p}(G):|u|_{A^\alpha_{p,q}(G)}:=\left\| \left( \lambda^{j\alpha} E(u,S_j)_{\Leb{p}(G)} \right)_{j\geq0} \right\|_{\leb{q}}<\infty \right\} , 
\end{equation}
for $0<p,q\leq\infty$, and $\alpha>0$, 
where $\lambda=2$ for red refinements and $\lambda=\sqrt[n]2$ for newest vertex bisections.
We will also use the shorthand notations 
\begin{equation}
A^\alpha_{p,q}(G) = A^\alpha_{p,q;m}(G) = A^\alpha_{p,q}(\{S_j\},G) .
\end{equation}
These spaces are quasi-Banach spaces with the quasi-norms $\|\cdot\|_{\Leb{p}(G)}+|\cdot|_{A^\alpha_{p,q}(G)}$.
Since $\Omega$ is bounded, it is clear that $A^\alpha_{p,q}(\Omega)\hookrightarrow A^{\alpha}_{p',q}(\Omega)$
for any $\alpha\geq0$, $0<q\leq\infty$ and $\infty\geq p>p'>0$.
We also have the lexicographical ordering:
$A^\alpha_{p,q}(\Omega)\hookrightarrow A^{\alpha'}_{p,q'}(\Omega)$ for $\alpha>\alpha'$ with any $0<q,q'\leq\infty$,
and
$A^\alpha_{p,q}(\Omega)\hookrightarrow A^\alpha_{p,q'}(\Omega)$ for $0<q<q'\leq\infty$.

It is no coincidence that the aforementioned embedding relations are identical to those among Besov spaces.
When reading the following theorem, keep in mind that $B^\alpha_{p,q;m+1}(\Omega)$ is the classical Besov space $B^\alpha_{p,q}(\Omega)$ for $\alpha<m+\max\{1,\frac1p\}$.

\begin{theorem}\label{t:multilevel-Besov}
We have ${B^\alpha_{p,q;m+1}(\Omega)} \hookrightarrow {A^\alpha_{p,q;m}(\Omega)}$ 
for $0<p,q\leq\infty$, and $\alpha>0$.
In the other direction, we have ${A^\alpha_{p,q;m}(\Omega)} \hookrightarrow {B^\alpha_{p,q;m+1}(\Omega)}$
for $0<p,q\leq\infty$, and $0<\alpha<1+\frac1p$.
\end{theorem}

\begin{proof}
We follow the standard approach.
The inclusion $B^\alpha_{p,q;m+1}(\Omega)\hookrightarrow A^\alpha_{p,q;m}(\Omega)$ is a direct consequence of \eqref{e:Jackson-Qj} with $p_0\leq p$
and the norm equivalence \eqref{e:Besov-norm-disc}.

For the second part, we start with the estimate
\begin{equation}\label{e:phi-z-est}
\omega_{m+1}(\phi_z,t)_p \lesssim \lambda^{-jn/p} \min\{1,(\lambda^{j}t)^{1+1/p}\},
\end{equation}
which holds for all nodal basis functions $\phi_z$ of $S_{j}$ and for all $j\geq0$.
This is Proposition 4.7 in \cite{GM13}, which also holds for $p=\infty$.
Hence for $0<p<\infty$ and for all $u_j=\sum_zb_z\phi_z\in S_{P_j}$, we infer
\begin{equation}
\begin{split}
\omega_{m+1}(u_j,t)_p^p 
&\lesssim \sum_{z} |b_z|^p \omega_{m+1}(\phi_z,t)_p^p 
\lesssim \sum_{z} |b_z|^p \, \lambda^{-jn} \min\{1,(\lambda^{j}t)^{p+1}\} \\
&\lesssim \min\{1,(\lambda^{j}t)^{p+1}\} \|u_j\|_{\Leb{p}(\Omega)}^p,
\end{split}
\end{equation}
where we have used the finite overlap property of the nodal basis functions, 
the $\Leb{p}$ stability of finite elements and the estimate $\|\phi_z\|_{\Leb{p}}\eqsim\lambda^{-jn/p}$.
The same ingredients are used to perform the corresponding computation for $p=\infty$, as
\begin{equation}
\begin{split}
\omega_{m+1}(u_j,t)_\infty
&\lesssim \max_{z} |b_z| \, \omega_{m+1}(\phi_z,t)_\infty
\lesssim \max_{z} |b_z| \min\{1,\lambda^{j}t\} \\
&\lesssim \min\{1,\lambda^{j}t\} \|u_j\|_{\Leb{\infty}(\Omega)} .
\end{split}
\end{equation}

Now we write $u=\sum_{j\geq0}(u_j-u_{j-1})$ with $u_j\in S_j$ a best approximation to $u$ from $S_j$ for $j\geq0$ and $u_{-1}=0$.
Note that the series converges in $\Leb{p}$ by \eqref{e:Jackson-Qj}.
With $p^*=\min\{1,p\}$, we have
\begin{equation}
\begin{split}
\omega_{m+1}(u,\lambda^{-k})_p^{p^*}
&\lesssim \sum_{j\geq0} \omega_{m+1}(u_j-u_{j-1},\lambda^{-k})_p^{p^*} \\
&\lesssim \sum_{j=0}^k \lambda^{(j-k)(1+1/p)p^*} \|u_j-u_{j-1}\|_{\Leb{p}(\Omega)}^{p^*} + \sum_{j=k+1}^\infty \|u_j-u_{j-1}\|_{\Leb{p}(\Omega)}^{p^*},
\end{split}
\end{equation}
and an application of the discrete Hardy inequality (Lemma \ref{e:Hardy-ineq-disc}) gives
\begin{equation}
|u|_{B^\alpha_{p,q;m+1}} \lesssim \left\| \left(\lambda^{j\alpha}\|u_j-u_{j-1}\|_{\Leb{p}(\Omega)} \right)_{j} \right\|_{\leb{q}},
\end{equation}
for $0<p,q\leq\infty$, and $0<\alpha<1+\frac1p$.
Finally, to go from $u_j-u_{j-1}$ to $u-u_j$ in the right hand side, we can apply the triangle inequality to $u_j-u_{j-1} = (u-u_{j-1}) - (u-u_j)$.
\end{proof}

\begin{figure}[ht]
\centering
\includegraphics{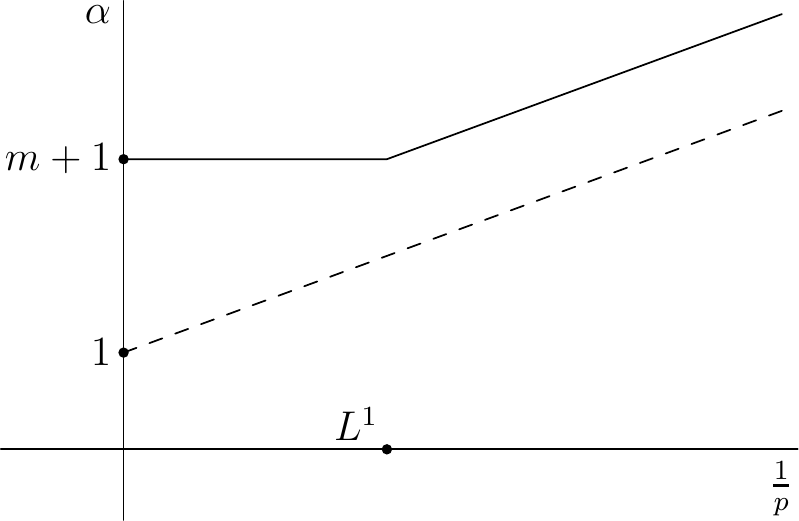}
\caption{
The inverse embedding $A^\alpha_{p,q;m}\hookrightarrow B^\alpha_{p,q;m+1}$ holds below the dashed line.
The direct embedding $B^\alpha_{p,q;m+1}\hookrightarrow A^\alpha_{p,q;m}$ holds without restriction, 
but the spaces $B^\alpha_{p,q;m+1}$ are nontrivial (and coincide with the classical Besov spaces $B^\alpha_{p,q}$) only below the solid line.}
\end{figure}

Notice the gap between the two inclusions: While $B^\alpha_{p,q;m+1}(\Omega)\hookrightarrow A^\alpha_{p,q;m}(\Omega)$ holds for all 
$\alpha>0$, the reverse inclusion is proved only for $0<\alpha<1+\frac1p$.
In fact, if $\alpha\geq1+\frac1p$ and $p<\infty$, the forward inclusion is strict:
Any function from $S_{j}$ would be an element of all $A^\alpha_{p,q;m}(\Omega)$,
but there are functions in $S_{j}$ that are not in $B^\alpha_{p,q;m+1}(\Omega)$,
because the estimate \eqref{e:phi-z-est} is saturated for small $t$.
This leads to the expectation that for large $\alpha$, the difference $A^\alpha_{p,q;m}(\Omega)\setminus B^\alpha_{p,q;m+1}(\Omega)$
should be ``skewed'' considerably depending on the initial mesh $P_0$.
We will not pursue this issue here,
but we conjecture that the Besov space $B^\alpha_{p,q;m+1}(\Omega)$ coincides with the intersection of all $A^\alpha_{p,q;m}(\Omega)$ 
as one considers all possible initial triangulations $P_0$.

We quote the following standard result, in order to assure the reader of the fact that the multilevel approximation spaces $A^\alpha_{p,q}(\Omega)$
coincide with the spaces $\hat B^\alpha_{p,q}(\Omega)$ considered in \cite{GM13}, cf. Definition 7.1 and Corollary 4.14 therein.

\begin{theorem}\label{t:multiscale-equiv}
Let $p_0\leq p\leq\infty$, $0<q\leq\infty$ and $\alpha>0$.
Then we have
\begin{equation}
\begin{split}
|u|_{A^\alpha_{p,q}(\Omega)} 
&\sim
\left\| \left( \lambda^{j\alpha} \|u-\tilde Q_ju\|_{\Leb{p}(\Omega)} \right)_{j\geq0} \right\|_{\leb{q}} \\
&\sim
\left\| \left( \lambda^{j\alpha} \|\tilde Q_{j+1}u-\tilde Q_ju\|_{\Leb{p}(\Omega)} \right)_{j\geq0} \right\|_{\leb{q}} ,
\end{split}
\end{equation}
for $u\in\Leb{p}(\Omega)$, where we have used the abbreviation $\tilde Q_j = \tilde Q_{P_j}$ for all $j$.
\end{theorem}

\begin{proof}
The first equivalence is immediate from \eqref{e:quasi-interp-global-best-Lagrange}.
The generalized triangle inequality
\begin{equation}
\|\tilde Q_{j+1}u-\tilde Q_ju\|_{\Leb{p}(\Omega)} \lesssim \|u-\tilde Q_ju\|_{\Leb{p}(\Omega)} + \|u-\tilde Q_{j+1}u\|_{\Leb{p}(\Omega)} ,
\end{equation}
implies one of the directions of the second equivalence,
while the other direction follows from applying the discrete Hardy inequality (Lemma \ref{e:Hardy-ineq-disc}) to
\begin{equation}
\|u-\tilde Q_ju\|_{\Leb{p}(\Omega)} \leq \left( \sum_{k=j}^\infty \|\tilde Q_{k+1}u-\tilde Q_ku\|_{\Leb{p}(\Omega)}^{p^*} \right)^{\frac1{p^*}} ,
\end{equation}
where $p^*=\min\{1,p\}$.
\end{proof}

The following technical result will be used later.

\begin{theorem}\label{t:interp-approx}
Let $0<\alpha_1<\alpha_2<\infty$ and $0<p,q,q_1,q_2\leq\infty$.
Then we have 
\begin{equation}\label{e:interp-approx}
[A^{\alpha_1}_{p,q_1}(G),A^{\alpha_2}_{p,q_2}(G)]_{\theta,q}=A^\alpha_{p,q}(G),
\end{equation}
for $\alpha=(1-\theta)\alpha_1+\theta\alpha_2$ and  $0<\theta<1$,
with the equivalence constants of quasi-norms depending only on the parameters $\alpha$, $\alpha_1$, $\alpha_2$, $p$, $q$, $q_1$ and $q_2$.
\end{theorem}

\begin{proof}
The equivalence \eqref{e:interp-approx} is standard, 
but we want to keep track of the equivalence constants.
So we sketch a proof here.
First, for $v\in S_m$, we observe the inverse inequality
\begin{equation}\label{e:inverse-est-approx}
|v|_{A^{\alpha_2}_{p,q_2}(G)}^{q_2} 
= \sum_{j=0}^{m-1} \lambda^{\alpha_2q_2j} E(v,S_j,G)_p^{q_2} 
\leq \|v\|_{\Leb{p}(G)}^{q_2}  \sum_{j=0}^{m-1} \lambda^{\alpha_2q_2j}
\leq \frac{\lambda^{\alpha_2q_2m}}{\lambda^{\alpha_2q_2}-1} \|v\|_{\Leb{p}(G)}^{q_2} .
\end{equation}
It is also true for $q_2=\infty$:
\begin{equation}
|v|_{A^{\alpha_2}_{p,\infty}(G)}
= \max_{0\leq j<m} \lambda^{\alpha_2j} E(v,S_j,G)_p
\leq \lambda^{\alpha_2m} \|v\|_{\Leb{p}(G)} .
\end{equation}
Another fact we will need is the following.
We have the generalized triangle inequality
\begin{equation}
|u+v|_{A^{\alpha_2}_{p,q_2}(G)} \leq c |u|_{A^{\alpha_2}_{p,q_2}(G)} + c|v|_{A^{\alpha_2}_{p,q_2}(G)} ,
\end{equation}
with $c\geq1$ depening only on $p$ and $q_2$.
Then the Aoki-Rolewicz theorem \cite[page 59]{BL76} implies that
\begin{equation}\label{e:-mu-triange-Apq2}
|v_1+\ldots+v_k|_{A^{\alpha_2}_{p,q_2}(G)}^\mu \leq 2 |v_1|_{A^{\alpha_2}_{p,q_2}(G)}^\mu +\ldots+ 2 |v_k|_{A^{\alpha_2}_{p,q_2}(G)}^\mu ,
\end{equation}
for any $v_1,\ldots,v_k\in A^{\alpha_2}_{p,q_2}(G)$, with $\mu$ given by $(2c)^\mu=2$.

With the abbreviation $K(u,t)=K(u,t;A^{\alpha_1}_{p,q_1}(G),A^{\alpha_2}_{p,q_2}(G))$, for $u\in A^{\alpha_1}_{p,q_1}(G)$, we have
\begin{equation}
K(u,\lambda^{-(\alpha_2-\alpha_1)m}) \leq |u-u_m|_{A^{\alpha_1}_{p,q_1}(G)} + \lambda^{-(\alpha_2-\alpha_1)m}|u_m|_{A^{\alpha_2}_{p,q_2}(G)} ,
\end{equation}
where $u_m\in S_m$ is an approximation satisfying $\|u-u_m\|_{\Leb{p}(G)}=E(u,S_m,G)_p$.
We estimate the first term in the right hand side as
\begin{equation}
\begin{split}
|u-u_m|_{A^{\alpha_1}_{p,q_1}(G)}^{q_1} 
&= \sum_{j=0}^m \lambda^{\alpha_1q_1j}\|u-u_m\|_{\Leb{p}(G)}^{q_1} + \sum_{j=m+1}^\infty \lambda^{\alpha_1q_1j}\|u-u_j\|_{\Leb{p}(G)}^{q_1} \\
&\lesssim \sum_{j=m}^\infty \lambda^{\alpha_1q_1j}\|u-u_j\|_{\Leb{p}(G)}^{q_1} ,
\end{split}
\end{equation}
with the implicit constant depending only on $\lambda^{\alpha_1q_1}$, and the second term as
\begin{equation}
\begin{split}
|u_m|_{A^{\alpha_2}_{p,q_2}(G)}^\mu 
&\leq 2\sum_{j=1}^m|u_j-u_{j-1}|_{A^{\alpha_2}_{p,q_2}(G)}^\mu
\lesssim \sum_{j=1}^m \lambda^{\alpha_2\mu j} \|u_j-u_{j-1}\|_{\Leb{p}(G)}^\mu \\
&\lesssim \sum_{j=0}^m \lambda^{\alpha_2\mu j} \|u-u_{j}\|_{\Leb{p}(G)}^\mu ,
\end{split}
\end{equation}
where we have used the $\mu$-triangle inequality \eqref{e:-mu-triange-Apq2} in the first step,
the inverse estimate \eqref{e:inverse-est-approx} in the second step, 
and the (generalized) triangle inequality for the $\Leb{p}$-quasi-norm in the third step.
Note that the implicit constants depend only on $\lambda^{\alpha_2q_2}$, $\lambda^{\alpha_2\mu}$, and $p$.
Putting everything together, we have
\begin{equation}
\begin{split}
K(u,\lambda^{-(\alpha_2-\alpha_1)m}) 
&\lesssim 
\left( \sum_{j=m}^\infty \lambda^{\alpha_1q_1j}\|u-u_j\|_{\Leb{p}(G)}^{q_1} \right)^{\frac1{q_1}} \\
&\quad + \lambda^{-(\alpha_2-\alpha_1)m}
\left( \sum_{j=0}^m \lambda^{\alpha_2\mu j} \|u-u_{j}\|_{\Leb{p}(G)}^\mu \right)^{\frac1\mu} ,
\end{split}
\end{equation}
and then the discrete Hardy inequalities (Lemma \ref{e:Hardy-ineq-disc}) give
\begin{equation}
\left\| [\lambda^{\gamma m} K(u,\lambda^{-(\alpha_2-\alpha_1)m})]_{m\geq0} \right\|_{\leb{q}} 
\lesssim 
\left\| [\lambda^{(\alpha_1+\gamma) m} \|u-u_{j}\|_{\Leb{p}(G)}]_{m\geq0} \right\|_{\leb{q}} ,
\end{equation}
for $0<\gamma<\alpha_2-\alpha_1$.
The left hand side of this inequality is the (quasi) norm for $[A^{\alpha_1}_{p,q_1}(G),A^{\alpha_2}_{p,q_2}(G)]_{\gamma/(\alpha_2-\alpha_1),q}$,
while the right hand side is the (quasi) norm for $A^{\alpha_1+\gamma}_{p,q}(G)$.

For the other direction, we start with 
\begin{equation}
\|u-u_j\|_{\Leb{p}(G)} \leq \|u-w_j-v_j\|_{\Leb{p}(G)} \lesssim \|u-v-w_j\|_{\Leb{p}(G)} + \|v-v_j\|_{\Leb{p}(G)} ,
\end{equation}
where $u\in A^{\alpha_1}_{p,q_1}(G)$ and $u_j\in S_j$ are as before,
and $v\in A^{\alpha_2}_{p,q_2}(G)$, $v_j,w_j\in S_j$ are arbitrary.
Note that the implicit constant depends only on $p$.
Optimizing over $v_j$ and $w_j$ gives
\begin{equation}
\min_{w_j\in S_j} \|u-v-w_j\|_{\Leb{p}(G)} \leq \lambda^{-\alpha_1j}|u-v|_{A^{\alpha_1}_{p,q_1}(G)} ,
\end{equation}
and
\begin{equation}
\min_{v_j\in S_j} \|v-v_j\|_{\Leb{p}(G)} \leq \lambda^{-\alpha_2j}|v|_{A^{\alpha_2}_{p,q_2}(G)} ,
\end{equation}
and substituting these back, we get
\begin{equation}
\begin{split}
\|u-u_j\|_{\Leb{p}(G)} 
&\lesssim \inf_{v\in A^{\alpha_2}_{p,q_2}(G)} \left( \lambda^{-\alpha_1j}|u-v|_{A^{\alpha_1}_{p,q_1}(G)} + \lambda^{-\alpha_2j}|v|_{A^{\alpha_2}_{p,q_2}(G)} \right) \\
&= \lambda^{-\alpha_1j} K(u,\lambda^{-(\alpha_2-\alpha_1)j}) .
\end{split}
\end{equation}
The proof is completed upon recalling the definition of $|\cdot|_{A^{\alpha}_{p,q}(G)}$.
\end{proof}

\subsection{Adaptive approximation}

In this subsection, we consider the approximation problem from adaptively generated Lagrange finite element spaces.
We study various approximation classes associated to the finite element spaces $S_P$, cf. \eqref{e:Lagrange-fem-space}.
In \cite{BDDP02,GM13}, among other things, it is proved that $B^{\alpha}_{q,q}(\Omega)\hookrightarrow\tstA^s_\infty(\Leb{p}(\Omega))$ with $s=\frac\alpha{n}$,
as long as $\frac\alpha{n}+\frac1p-\frac1q>0$ and $0<\alpha<m+ \max\{1,\frac1q\}$.
In the other direction, 
the same references give $\tstA^s_q(\Leb{p}(\Omega))\hookrightarrow A^{\alpha}_{q,q}(\Omega)$ for $s=\frac\alpha{n}=\frac1q-\frac1p>0$ and $0<p,q<\infty$.

Below we complement these results by establishing direct embeddings of the form $A^{\alpha}_{q,q}(\Omega)\hookrightarrow\tstA^s_\infty(\Leb{p}(\Omega))$.
This is a genuine improvement, since $A^{\alpha}_{q,q}(\Omega)\supsetneq B^{\alpha}_{q,q}(\Omega)$ for $\alpha\geq1+\frac1q$.
Moreover, it seems natural to relate adaptive approximation to multilevel approximation first, 
and then bring in the relationships between multilevel approximation and Besov spaces.
We also remark that while the existing results are only for the newest vertex bisection procedure and conforming triangulations, 
we deal with possibly nonconforming triangulations, and therefore are able to handle the red refinement procedure, 
as well as newest vertex bisections without the conformity requirement.

\begin{theorem}\label{t:direct-Lp}
Let $0<q\leq p\leq\infty$ and $\alpha>0$ satisfy $\frac\alpha{n}+\frac1p-\frac1q>0$ and $q<\infty$.
Then for any $0<p_0<q$ (Recall that $\tilde Q_P$ depends on $p_0$), we have
\begin{equation}\label{e:direct-Lagrange}
\|u-\tilde Q_Pu\|_{\Leb{p}(\Omega)} \lesssim \left( \sum_{\tau\in P} |\tau|^{p\delta} |u|_{A^\alpha_{q,q}(\hat\tau)}^p \right)^{\frac1p} ,
\qquad u\in A^\alpha_{q,q}(\Omega) , \quad P\in\tstP ,
\end{equation}
where $\delta=\frac\alpha{n}+\frac1p-\frac1q$.
In particular, we have $A^{\alpha}_{q,q}(\Omega)\hookrightarrow\tstA^s(\Leb{p}(\Omega))$ with $s=\frac\alpha{n}$.
\end{theorem}

\begin{proof}
We have the sub-additivity property
\begin{equation}
\sum_{k} \|u\|_{A^\alpha_{q,q}(P_k(\tau_k))}^q \lesssim \|u\|_{A^\alpha_{q,q}(\Omega)}^q,
\end{equation}
for $0<q<\infty$ and for any finite sequences $\{P_k\}\subset\tstP$ and $\{\tau_k\}$, with $\tau_k\in P_k$ and $\{\tau_k\}$ non-overlapping.
Recall that  $\hat\tau=P(\tau)$ is the support extension of $\tau$, as defined in \eqref{e:supp-ext}.
Therefore the estimate \eqref{e:direct-Lagrange} would imply the second statement by Theorem \ref{t:direct-std}.

We shall prove \eqref{e:direct-Lagrange}. 
Every element $\tau\in P$ of any partition $P\in\tstP$ is an element of a unique $P_j$,
with the number $j$ counting how many refinements one needs in order to arrive at $\tau$.
We call $j$ the {\em generation} or the {\em level} of $\tau$, and write $j=[\tau]$.
We will also need $j(\tau)=\min\{[\sigma]:\sigma\in P,\,\sigma\subset\hat\tau\}$.
Note that $|\tau|\sim\lambda^{-n[\tau]}\sim\lambda^{-nj(\tau)}$
and $S_{j(\tau)}|_{\hat\tau}\subset S_P|_{\hat\tau}$.
By invoking \eqref{e:quasi-interp-local-best-Lagrange}, we infer
\begin{equation}
\begin{split}
\|u - \tilde Q_{P} u\|_{\Leb{p}(\Omega)}^p 
&= \sum_{\tau\in P} \|u - \tilde Q_{P} u\|_{\Leb{p}(\tau)}^p 
\lesssim \sum_{\tau\in P} \inf_{v\in S_P} \|u-v\|_{\Leb{p}(\hat\tau)}^p \\
&\leq \sum_{\tau\in P} \|u-u_{j(\tau)}\|_{\Leb{p}(\hat\tau)}^p ,
\end{split}
\end{equation}
where $u_j\in S_j$ ($j\geq0$) is an approximation (that may depend on $\tau$) satisfying 
\begin{equation}\label{e:def-uj-pf}
\|u-u_j\|_{\Leb{q}(\hat\tau)}\leq c E(u,S_j)_{\Leb{q}(\hat\tau)},
\end{equation}
with some constant $c\geq1$.
The same is true for $p=\infty$ with obvious modifications.
For an individual term in the right hand side, with $p^*=\min\{1,p\}$, we have
\begin{equation}
\begin{split}
\|u-u_{j(\tau)}\|_{\Leb{p}(\hat\tau)}^{p^*} 
&\leq
\sum_{j=j(\tau)}^\infty \|u_{j+1}-u_j\|_{\Leb{p}(\hat\tau)}^{p^*} 
\lesssim
\sum_{j=j(\tau)}^\infty \lambda^{(\frac1q-\frac1p)jnp^*}\|u_{j+1}-u_j\|_{\Leb{q}(\hat\tau)}^{p^*} \\
&\lesssim
\sum_{j=j(\tau)}^\infty \lambda^{(\frac1q-\frac1p)jnp^*}\|u-u_j\|_{\Leb{q}(\hat\tau)}^{p^*} ,
\end{split}
\end{equation}
where we have estimated $u-u_{j(\tau)}$ as a telescoping sum in the first step,
and used the estimate $\lambda^{jn/p}\|v\|_{\Leb{p}(\hat\tau)}\sim\lambda^{jn/q}\|v\|_{\Leb{q}(\hat\tau)}$ for $v\in S_{j+1}$ in the second step.
We continue by noting the relation $\frac1q-\frac1p=\frac\alpha{n}-\delta$, which yields
\begin{equation}\label{e:u-uj-pf}
\begin{split}
\|u-u_{j(\tau)}\|_{\Leb{p}(\hat\tau)}^{p^*} 
&\lesssim
\sum_{j=j(\tau)}^\infty \lambda^{-j\delta np^*} \lambda^{j\alpha p^*} \|u-u_j\|_{\Leb{q}(\hat\tau)}^{p^*} \\
&\leq
\lambda^{-j(\tau)\delta np^*} \sum_{j=j(\tau)}^\infty \lambda^{j\alpha p^*} \|u-u_j\|_{\Leb{q}(\hat\tau)}^{p^*} \\
&\lesssim
|\tau|^{\delta p^*} |u|_{A^\alpha_{q,p^*}(\hat\tau)}^{p^*} ,
\end{split}
\end{equation}
by \eqref{e:def-uj-pf}.
This establishes the theorem for $q\leq1$, in which case we have $A^\alpha_{q,q}(\hat\tau)\hookrightarrow A^\alpha_{q,p^*}(\hat\tau)$.

If $q>1$, choose $0<\alpha_1<\alpha<\alpha_2$ satisfying $\alpha=\frac{\alpha_1+\alpha_2}{2}$ and $\delta_i=\frac{\alpha_i}{n}+\frac1p-\frac1q>0$ for $i=1,2$.
Moreover, we put $u_j=Q_{P_j}^{(\hat\tau)}u$, where $Q_{P_j}^{(\hat\tau)}:\Leb{1}(\hat\tau)\to S_j|_{\hat\tau}$ is the quasi-interpolation operator defined in \eqref{e:quasi-interpolator-std},
with $\hat\tau$ playing the role of $\Omega$.
Then Lemma \ref{l:quasi-interp-std} guarantees the property \eqref{e:def-uj-pf} with $c$ depending only on global geometric properties of $\tstP$.
In particular, $c$ is bounded independently of $\tau$.
Thus \eqref{e:u-uj-pf} gives
\begin{equation}
\|u-u_{j(\tau)}\|_{L^{p}(\hat\tau)} \lesssim |\tau|^{\delta_i} |u|_{A^{\alpha_i}_{q,1}(\hat\tau)} ,
\end{equation}
for $i=1,2$.
Since the operators $Q_{P_j}^{(\hat\tau)}$ are linear, so is the map $u\mapsto u-u_{j(\tau)}$,
and hence interpolation and Theorem \ref{t:interp-approx} yield
\begin{equation}
\|u-u_{j(\tau)}\|_{L^{p}(\hat\tau)} 
\lesssim |\tau|^{(\delta_1+\delta_2)/2} |u|_{[A^{\alpha_1}_{q,1}(\hat\tau),A^{\alpha_2}_{q,1}(\hat\tau)]_{1/2,p}}
\lesssim |\tau|^{\delta} |u|_{A^{\alpha}_{q,p}(\hat\tau)}
\lesssim |\tau|^{\delta} |u|_{A^{\alpha}_{q,q}(\hat\tau)} ,
\end{equation}
with the implicit constants depending only on global geometric properties of $\tstP$ and on the indices of the spaces involved.
This completes the proof.
\end{proof}

\begin{figure}[ht]
\centering
\begin{subfigure}{0.45\textwidth}
\includegraphics[width=0.85\textwidth]{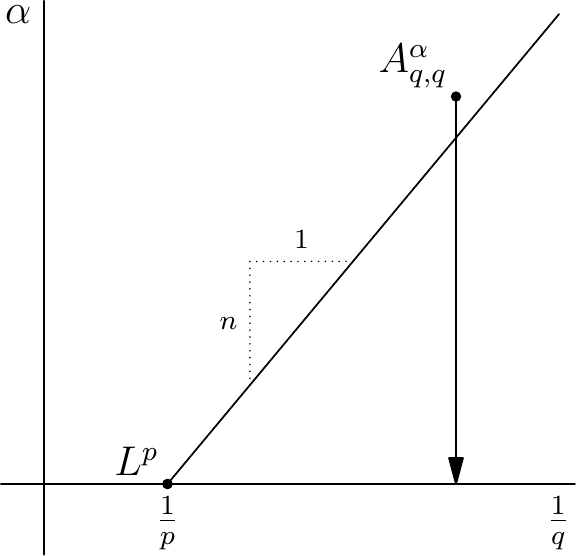}
\subcaption{If the space $A^\alpha_{q,q}$ is located above the solid line, 
we have $A^\alpha_{q,q}\subset\tstA^s(\Leb{p})$ with $s=\frac\alpha{n}$.}
\end{subfigure}
\qquad
\begin{subfigure}{0.45\textwidth}
\includegraphics[width=0.9\textwidth]{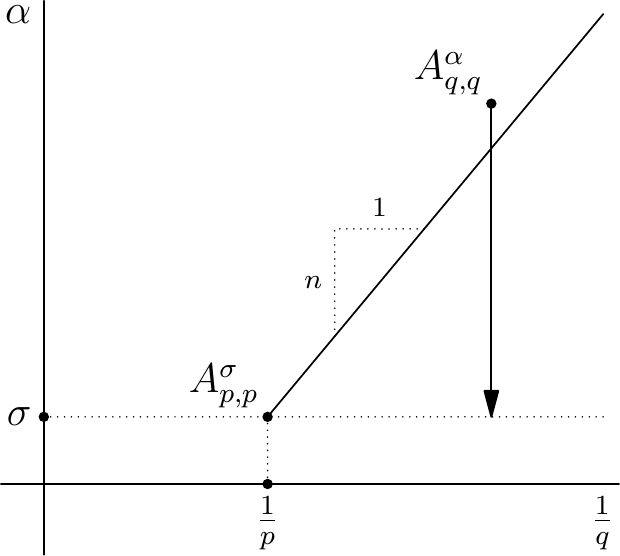}
\subcaption{If the space $A^\alpha_{q,q}$ is located above the solid line, 
we have $A^\alpha_{q,q}\subset\tstA^s(A^\sigma_{p,p})$ with $s=\frac{\alpha-\sigma}{n}$.}
\end{subfigure}
\caption{Illustration of Theorem \ref{t:direct-Lp} and Theorem \ref{t:direct-App}.}
\label{f:direct-embeddings-std}
\end{figure}

Now we look at adaptive approximation in the space $A^\sigma_{p,p}(\Omega)$.
Recall from \cite{GM13} that $\tstA^s_q(A^\sigma_{p,p}(\Omega))\hookrightarrow A^{\alpha}_{q,q}(\Omega)$ for $s=\frac{\alpha-\sigma}{n}=\frac1q-\frac1p>0$ and $0<p,q<\infty$.

\begin{theorem}\label{t:direct-App}
Let $0<q\leq p\leq\infty$, and $\alpha,\sigma>0$ satisfy $\frac{\alpha-\sigma}{n}+\frac1p-\frac1q>0$ and $q<\infty$.
Then for any $0<p_0<q$, we have
\begin{equation}
\|u-\tilde Q_Pu\|_{A^\sigma_{p,p}(\Omega)} \lesssim \left( \sum_{\tau\in P} |\tau|^{p\delta} |u|_{A^\alpha_{q,q}(\hat\tau)}^p \right)^{\frac1p} ,
\qquad u\in A^\alpha_{q,q}(\Omega) , \quad P\in\tstP ,
\end{equation}
with $\delta=\frac{\alpha-\sigma}{n}+\frac1p-\frac1q$.
In particular, we have $A^{\alpha}_{q,q}(\Omega)\hookrightarrow\tstA^s(A^\sigma_{p,p}(\Omega))$ with $s=\frac{\alpha-\sigma}{n}$.
\end{theorem}

\begin{proof}
With $v=u-\tilde Q_Pu$ and $\tilde Q_j=\tilde Q_{P_j}$, we have
\begin{equation}\label{e:original-sum-pf}
\|v\|_{A^\sigma_{p,p}(\Omega)}
\leq \left( \sum_{j\geq0} \lambda^{j\sigma p} \|v-\tilde Q_jv\|_{\Leb{p}(\Omega)}^p \right)^{\frac1p}
= \left( \sum_{\tau\in P} \sum_{j\geq0} \lambda^{j\sigma p} \|v-\tilde Q_jv\|_{\Leb{p}(\tau)}^p \right)^{\frac1p} ,
\end{equation}
with the usual modification for $p=\infty$.
Let $j(\tau)=\max\{[\sigma]:\sigma\in P,\,\sigma\subset\tau\}$ for $\tau\in P$,
as in the preceding proof,
with $[\sigma]$ denoting the generation number (or the level) of $\sigma$. 
Then for $\tau\in P$ and $j\geq j(\tau)$ we have $S_P|_{\hat\tau}\subset S_j|_{\hat\tau}$, and hence
\begin{equation}
\tilde Q_j(u-\tilde Q_Pu) = \tilde Q_ju-\tilde Q_Pu
\quad \textrm{on } \tau ,
\end{equation}
by the linearity property \eqref{e:quasi-interpolator-linearity}.
This implies that $v-\tilde Q_jv=u-\tilde Q_ju$ on $\tau$, for all $j\geq j(\tau)$.
Now, 
proceeding exactly as in the preceding proof, with $p^*=\min\{1,p\}$, we infer
\begin{equation}
\begin{split}
\|u-\tilde Q_ju\|_{\Leb{p}(\tau)}^{p^*} 
&\leq
\sum_{k=j}^\infty \|\tilde Q_{k+1}u-\tilde Q_ku\|_{\Leb{p}(\tau)}^{p^*} 
\lesssim
\sum_{k=j}^\infty \lambda^{(\frac1q-\frac1p)knp^*} \|\tilde Q_{k+1}u-\tilde Q_ku\|_{\Leb{q}(\tau)}^{p^*}  \\
&\lesssim
\sum_{k=j}^\infty \lambda^{(\frac1q-\frac1p)knp^*} \|u-\tilde Q_ku\|_{\Leb{q}(\tau)}^{p^*} .
\end{split}
\end{equation}
Then the discrete Hardy inequality yields
\begin{equation}
\begin{split}
\sum_{j\geq j(\tau)} \lambda^{j\sigma q}\|u-\tilde Q_ju\|_{\Leb{p}(\tau)}^{q} 
&\lesssim
\sum_{k\geq j(\tau)} \lambda^{k\sigma q} \lambda^{(\frac1q-\frac1p)knq} \|u-\tilde Q_ku\|_{\Leb{q}(\tau)}^{q} \\
&\leq
\lambda^{-\delta n q j(\tau)} \sum_{k=j(\tau)}^\infty \lambda^{k\alpha q} \|u-\tilde Q_ku\|_{\Leb{q}(\tau)}^{q} \\
&\lesssim 
|\tau|^{\delta q} |u|_{A^\alpha_{q,q}(\tau)}^q ,
\end{split}
\end{equation}
where we have taken into account the relation $\frac{\sigma}{n}+\frac1q-\frac1p=\frac{\alpha}{n}-\delta$.
Notice that the discrete Hardy inequality made the use of interpolation unnecessary, to compare the present arguments with the proof of the preceding theorem.
This takes care of one of the sums (or maximums) when we split the sum in the right hand side of \eqref{e:original-sum-pf} into two sums according to $j<j(\tau)$ or $j\geq j(\tau)$.
We rewrite the other sum (or maximum) as
\begin{equation}\label{e:arg-pf}
\begin{split}
\left( \sum_{\tau\in P} \sum_{\{j<j(\tau)\}} \lambda^{j\sigma p}\|v-\tilde Q_jv\|_{\Leb{p}(\tau)}^{p} \right)^{\frac1p}
&=
\left( \sum_{j\geq0} \sum_{\{\tau\in P:j(\tau)>j\}} \lambda^{j\sigma p}\|v-\tilde Q_jv\|_{\Leb{p}(\tau)}^{p} \right)^{\frac1p} \\
&=
\left( \sum_{j\geq0} \lambda^{j\sigma p} \|v-\tilde Q_jv\|_{\Leb{p}(\Omega_j)}^{p} \right)^{\frac1p} ,
\end{split}
\end{equation}
where $\Omega_j=\bigcup\{\tau\in P:j(\tau)>j\}$.
Note that $\Omega_j\supset\Omega_j^0$ with $\Omega_j^0=\bigcup\{\tau\in P:[\tau]>j\}$,
and that $\Omega_j^0$ consists of triangles from $P_j$, in the sense that there is $R_j^0\subset P_j$ such that $\Omega_j^0=\bigcup\{\tau\in R_j^0\}$ up to a zero measure set.
The triangles $\tau\in P$ with $\tau\not\subset\Omega_j^0$ are at the level $j$ or less, and 
hence there is $R_j\subset P_j$ such that $\Omega_j=\bigcup\{\tau\in R_j\}$ up to a zero measure set.
Now, by the stability property \eqref{e:quasi-interp-stable}, we get
\begin{equation}
\|v-\tilde Q_jv\|_{\Leb{p}(\Omega_j)}
\lesssim
\|v\|_{\Leb{p}(\Omega_j)} + \|\tilde Q_jv\|_{\Leb{p}(\Omega_j)}
\lesssim
\|v\|_{\Leb{p}(\hat\Omega_j)} ,
\end{equation}
where $\hat\Omega_j = \bigcup\{\tau\in P_j:\bar\tau\cap\Omega_j\neq\varnothing\}$.
Obviously, $\hat\Omega_j$ is a subset of $\hat\Omega_j'=\bigcup\{\tau\in P:\bar\tau\cap\Omega_j\neq\varnothing\}$,
that can also be described as $\hat\Omega_j'=\bigcup\{\tau\in P:j^2(\tau)>j\}$, with
$j^2(\tau)=\max\{j(\sigma):\sigma\in P,\,\bar\sigma\cap\bar\tau\neq\varnothing\}$ for $\tau\in P$.
All this yields
\begin{equation}
\begin{split}
\left( \sum_{\tau\in P} \sum_{\{j<j(\tau)\}} \lambda^{j\sigma p}\|v-\tilde Q_jv\|_{\Leb{p}(\tau)}^{p} \right)^{\frac1p}
&\lesssim
\left( \sum_{\tau\in P} \sum_{\{j<j^2(\tau)\}} \lambda^{j\sigma p}\|u-\tilde Q_Pu\|_{\Leb{p}(\tau)}^{p} \right)^{\frac1p} \\
&\lesssim
\left( \sum_{\tau\in P} |\tau|^{\sigma p/n} \|u-\tilde Q_Pu\|_{\Leb{p}(\tau)}^{p} \right)^{\frac1p} ,
\end{split}
\end{equation}
where we have taken into account the geometric growth of $\lambda^{j\sigma p}$ in $j$,
and the fact that $\lambda^{j^2(\tau)}\sim|\tau|^{1/n}$.
Then once we recall from the proof of the preceding theorem that
\begin{equation}
\|u-\tilde Q_Pu\|_{\Leb{p}(\tau)}
\lesssim
|\tau|^{\delta'} |u|_{A^\alpha_{q,q}(\hat\tau)}
\end{equation}
with $\delta'=\frac\alpha{n}+\frac1p-\frac1q=\delta-\frac\sigma{n}$,
the proof is complete.
\end{proof}

\subsection{Discontinuous piecewise polynomials}
\label{ss:disc}

All that has been said on multilevel and adaptive approximation for continuous Lagrange finite elements have analogues in 
the world of discontinuous polynomials subordinate to triangulations.
The theory is in fact much simpler due to the absence of the continuity requirement across elements.
Thus we will state here the relevant results and only sketch or omit the proofs.

The notations $\tstP$, $\{P_j\}$, etc., will mean the same things as before.
For $P\in\tstP$, let
\begin{equation}
\bar S_P = \bar S^d_P = \{v\in \Leb{\infty}(\Omega):v|_\tau\in\Pol_d\,\forall\tau\in P\} ,
\end{equation}
where $d$ is a nonnegative integer, and let $\bar S_j=\bar S_{P_j}$ for all $j$.
Then with $G\subset\Omega$ a domain consisting of elements from some $P_j$, we define the {multilevel approximation spaces} 
$A^\alpha_{p,q}(\{\bar S_j\},G)$ by \eqref{e:multilevel-approx}, with the sequence $\{\bar S_j\}$ replacing $\{S_j\}$.
We will also use the shorthand notations 
\begin{equation}
\bar A^\alpha_{p,q}(G) = \bar A^\alpha_{p,q;d}(G) = A^\alpha_{p,q}(\{\bar S_j\},G) .
\end{equation}

The analogue of Theorem \ref{t:multilevel-Besov} is the following.

\begin{theorem}\label{t:multilevel-Besov-disc}
We have ${B^\alpha_{p,q;d+1}(\Omega)} \hookrightarrow {\bar A^\alpha_{p,q;d}(\Omega)}$ 
for $0<p,q\leq\infty$, and $\alpha>0$.
In the other direction, we have ${\bar A^\alpha_{p,q;d}(\Omega)} \hookrightarrow {B^\alpha_{p,q;d+1}(\Omega)}$
for $0<p,q\leq\infty$, and $0<\alpha<\frac1p$.
\end{theorem}

Note that due to the lack of continuity the inverse inclusion holds in a very small range of indices.
We also have the analogue of Theorem \ref{t:interp-approx}.

\begin{theorem}
Let $0<\alpha_1<\alpha_2<\infty$ and $0<p,q,q_1,q_2\leq\infty$.
Then we have 
\begin{equation}
[\bar A^{\alpha_1}_{p,q_1}(G),\bar A^{\alpha_2}_{p,q_2}(G)]_{\theta,q} = \bar A^\alpha_{p,q}(G),
\end{equation}
for $\alpha=(1-\theta)\alpha_1+\theta\alpha_2$ and  $0<\theta<1$,
with the equivalence constants of quasi-norms depending only on the parameters $\alpha$, $\alpha_1$, $\alpha_2$, $p$, $q$, $q_1$ and $q_2$.
\end{theorem}

Finally, we want to record some results on adaptive approximation by discontinuous polynomials subordinate to the partitions in $\tstP$.
Given $0<p\leq\infty$, $\theta\in\R$, and $P\in\tstP$, we define the norm
\begin{equation}\label{e:Leb-disc}
\|u\|_{\Leb[\theta]{p}(\Omega)} = \left( \sum_{\tau\in P} |\tau|^{\frac{\theta p}n} \|u\|_{\Leb{p}(\tau)}^p \right)^{\frac1p} ,
\qquad \textrm{for} \quad u\in\Leb{p}(\Omega) ,
\end{equation}
with the obvious modification for $p=\infty$, and denote by $\Leb[\theta]{p}(\Omega)$ the space $\Leb{p}(\Omega)$ equipped with this norm.
Then we define the approximation class $\bar\tstA^s_{q;d}(\Leb[\theta]{p}(\Omega))$ exactly as $\tstA^s_q(\Leb{p}(\Omega))$,
by replacing $S^m_P$ with $\bar S^d_P$, and by using the distance function
\begin{equation}\label{e:rho-disc}
\rho(u,v,P)  = \|u-v\|_{\Leb[\theta]{p}(\Omega)} .
\end{equation}
More precisely, recalling the definition \eqref{e:A-rho-def}, let
\begin{equation}\label{e:adapt-class-disc}
\bar\tstA^s_{q;d}(\Leb[\theta]{p}(\Omega)) = \tstA^s_q(\rho,\tstP,\{\bar S^d_P\}) ,
\end{equation}
with $\rho$ given by \eqref{e:rho-disc}.
It is for later reference that we have introduced the mesh dependent weight in the distance function.
We write $\bar\tstA^s_{q;d}(\Leb{p}(\Omega))=\bar\tstA^s_{q;d}(\Leb[0]{p}(\Omega))$.

We have the following direct embedding result.

\begin{theorem}\label{t:direct-Lp-disc}
Let $0<q\leq p\leq\infty$, $\alpha>0$ and $\theta\geq0$ satisfy $\frac\alpha{n}+\frac1p-\frac1q>0$ and $q<\infty$, with $\frac\alpha{n}+\frac1p-\frac1q=0$ allowed if $\theta>0$.
Then we have $\bar A^{\alpha}_{q,q;d}(\Omega)\hookrightarrow\bar\tstA^s_{\infty;d}(\Leb[\theta]{p}(\Omega))$ with $s=\frac{\alpha+\theta}{n}$.
\end{theorem}

\begin{proof}
Let $u\in\Leb{p}(\Omega)$ and let $P\in\tstP$.
Then with $\Pi_P$ the projection operator defined in \eqref{e:quasi-interpolator-disc} with $m:=d$, we have
\begin{equation}
\|u - \Pi_P u\|_{\Leb[\theta]{p}(\Omega)}
= \left( \sum_{\tau\in P} |\tau|^{\frac{\theta p}n} \|u - \Pi_{P} u\|_{\Leb{p}(\tau)}^p \right)^{\frac1p}
\lesssim \left( \sum_{\tau\in P} |\tau|^{\frac{\theta p}n} \inf_{v\in \bar S_P} \|u-v\|_{\Leb{p}(\tau)}^p \right)^{\frac1p} .
\end{equation}
Now proceeding exactly as in the proof of Theorem \ref{t:direct-Lp}, we get
\begin{equation}
\|u - \Pi_P u\|_{\Leb[\theta]{p}(\Omega)}
\lesssim \left( \sum_{\tau\in P} |\tau|^{\frac{\theta p}n} |\tau|^{\delta p} |u|_{\bar A^{\alpha}_{q,q;d}(\tau)} \right)^{\frac1p} ,
\end{equation}
with $\delta=\frac\alpha{n}+\frac1p-\frac1q$.
Then an application of Theorem \ref{t:direct-std} finishes the proof.
\end{proof}

\begin{figure}[ht]
\centering
\includegraphics[width=0.5\textwidth]{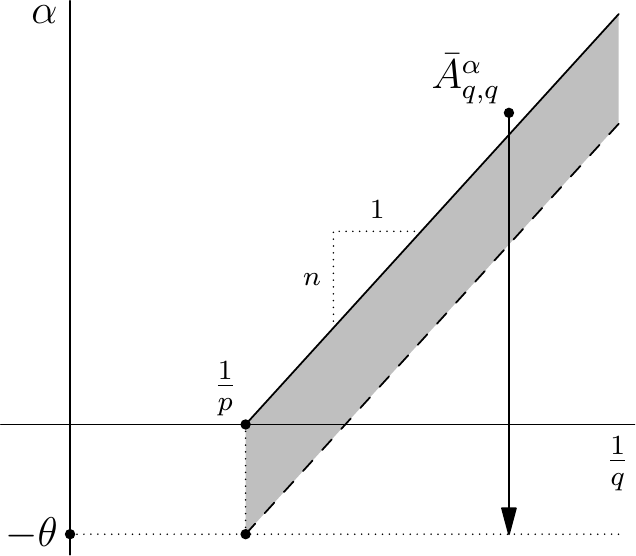}
\caption{Illustration of Theorem \ref{t:direct-Lp-disc} and Theorem \ref{t:inverse-Lp-disc}.
If the space $\bar A^\alpha_{q,q}$ is located above or on the solid line,
then $\bar A^\alpha_{q,q}\subset\bar\tstA^s(\Leb[\theta]{p})$ with $s=\frac{\alpha+\theta}n$.
It is as if the approximation is taking place in a space such as $B^{-\theta}_{p,p}$,
but instead of $\frac{\alpha+\theta}{n}>\frac1q-\frac1p$ (dashed line) we have the condition $\frac\alpha{n}\geq\frac1q-\frac1p$ (solid line).
On the other hand, the inverse embedding takes the form $\bar\tstA^s_{q}(\Leb[\theta]{p})\cap\Leb{p}\subset\bar A^{\alpha}_{q,q}$,
which holds on the part of the dashed line with $\alpha>0$.}
\label{f:direct-embeddings-disc}
\end{figure}

We close this section by stating an inverse embedding theorem (A proof can be found in arXiv version 1 of the current paper).

\begin{theorem}\label{t:inverse-Lp-disc}
Let $0<q\leq p<\infty$, $\alpha,\theta>0$, and let $s=\frac{\alpha+\theta}{n}=\frac1q-\frac1p$.
Then we have $\bar\tstA^s_{q;d}(\Leb[\theta]{p}(\Omega))\cap\Leb{p}(\Omega)\subset\bar A^{\alpha}_{q,q;d}(\Omega)$.
\end{theorem}

\section{Second order elliptic problems}
\label{s:2nd-order}

\subsection{Introduction}

In this section, we will apply the abstract theory of Section \ref{s:general} to second order elliptic boundary value problems.
As far as the domain $\Omega$ and the family of triangulations $\tstP$ are concerned,
we will keep the setting of the previous section intact.
In particular, we fix a refinement rule, which is either the newest vertex bisection or the red refinement,
and assume that the family $\tstP$ satisfies the admissibility criterion \eqref{e:finite-support}.

Let $\Gamma\subset\partial\Omega$ be an open piece (or the whole) of the boundary, consisting of faces of the initial triangulation $P_0$.
For $P\in\tstP$, the space $S_P$ will be the Lagrange finite element space of continuous piecewise polynomials of degree not exceeding $m$,
with the homogeneous Dirichlet condition on $\Gamma$.
We also define $H^1_\Gamma(\R^n)$ as the closure of $\tstD(\R^n\setminus\overline\Gamma)$ in $H^1(\R^n)$,
and $H^1_\Gamma=H^1_\Gamma(\Omega)$ as the restriction of functions from $H^1_\Gamma(\R^n)$ to $\Omega$.
Note that $S_P=H^1_\Gamma\cap S^m_P$.

The operator $T$ is given as
\begin{equation}
Tu = - a_{ij} \partial_i \partial_j u + b_{k} \partial_k u + cu,
\end{equation}
where the repeated indices are summed over.
The coefficients $a_{ij}$ are Lipschitz continuous, and $b_{k}, c\in \Leb{\infty}(\Omega)$. 
The problem we consider is to find $u\in H^1_\Gamma$ satisfying
\begin{equation}\label{e:2nd-order-var-form}
\langle Tu, v\rangle  = \langle f, v\rangle,
\qquad\textrm{for all}\quad
v\in H^1_\Gamma.
\end{equation}
Here $\langle\cdot,\cdot\rangle$ is the duality pairing between $(H^1_\Gamma)'$ and $H^1_\Gamma$,
and $f\in \Leb{2}(\Omega) \hookrightarrow (H^1_\Gamma)'$ is given.
This is of course the variational formulation of the mixed Dirichlet-Neumann problem with the homogenous Dirichlet data on $\Gamma$.
We will also denote by $T:H^1_\Gamma\to(H^1_\Gamma)'$ the operator defined in \eqref{e:2nd-order-var-form}.

Given an $n$-simplex $\tau$, let us denote by $E_\tau$ the union of the $(n-1)$-dimensional open faces of $\tau$.
In other words, $E_\tau$ is the boundary of $\tau$ with all but the $(n-1)$-dimensional faces removed.
Then for $P\in\tstP$, let
\begin{equation}
E_P=\{E_\tau\cap E_\sigma\cap(\bar\Omega\setminus\Gamma):\tau,\sigma\in P\} ,
\end{equation}
be the set of faces not intersecting the Dirichlet piece $\Gamma$.
Note that if $P$ is nonconforming, then only the faces of the ``smaller'' simplices go into $E_P$.
Given $P\in\tstP$, $u\in T^{-1}(\Leb{2}(\Omega))$, and $v\in S_P$, define the {\em element residual} $r_\tau=(Tu-Tv)|_\tau$ for $\tau\in P$,
and the {\em edge residual} $r_e\in \Leb{2}(e)$ for $e\in E_P$ as the jump of the normal component of the vector field $a_{ij}\partial_jv$ across the edge $e$.
Finally, we define the {\em a posteriori} error estimator
\begin{equation}\label{e:2nd-order-err-est}
(\eta(u,v,P))^2 = \sum_{\tau\in P} h_\tau^{2}\|r_\tau\|_{\Leb{2}(\tau)}^2 + \sum_{e\in E_P} h_e \|r_e\|_{\Leb{2}(e)}^2 .
\end{equation}
A typical adaptive finite element method that uses \eqref{e:2nd-order-err-est} as its error indicator converges optimally
with respect to the approximation classes $\tstA^s(\eta)$,
in the sense that if the solution $u$ of the problem \eqref{e:2nd-order-var-form} satisfies $u\in\tstA^s(\eta)$ for some $s>0$,
then the adaptive method reduces the quantity $\eta(u,u_P,P)$ with the rate $s$, where $u_P$ is the Galerkin approximation of $u$ from $S_P$,
cf. \cite*{FFP14}.
Moreover, it is well known that the estimator \eqref{e:2nd-order-err-est} is equivalent to the {\em total error}
\begin{equation}\label{e:2nd-order-tot-err}
(\rho_d(u,v,P))^2 = \|u-v\|_{H^{1}}^2 + (\osc_d(u,v,P))^2,
\end{equation}
when $v=u_P$ and for any fixed $d\geq m-2$,
where the {\em oscillation} is defined as
\begin{equation}
(\osc_d(u,v,P))^2 = \sum_{\tau\in P} h_\tau^{2}\|(1-\Pi_\tau)r_\tau\|_{\Leb{2}(\tau)}^2 + \sum_{e\in E_P} h_e \|(1-\Pi_e) r_e\|_{\Leb{2}(e)}^2 ,
\end{equation}
with $\Pi_\tau:\Leb{2}(\tau)\to\Pol_d$ and $\Pi_e:\Leb{2}(e)\to\Pol_{d+1}$ being $\Leb{2}$-orthogonal projections onto polynomial spaces,
see e.g., \cite*{NSV09}.
Optimality of adaptive finite element methods with respect to the approximation classes $\tstA^s(\rho_d)$ has also been proved, cf. \cite{CKNS08}.
Ideally, one would like to have optimality with respect to the classes $\tstA^s(H^1_\Gamma)$ that correspond to the energy error.
In particular, it is conceivable that for certain functions $u$, the energy error $\|u-u_P\|_{H^1}$ decays faster than the oscillation 
$\osc_d(u,u_P,P)$, so that the class $\tstA^s(H^1_\Gamma)$ is strictly larger than both $\tstA^s(\rho_d)$ and $\tstA^s(\eta)$.
However, if the error estimator \eqref{e:2nd-order-err-est} is the only source of information used by the algorithm in its stopping criterion
(or in the marking of triangles for refinement),
then it is clear that one has to reduce the oscillation anyway.
It appears therefore that the approximation classes $\tstA^s(\rho_d)$ and $\tstA^s(\eta)$ are completely natural  from the perspective of adaptive finite element methods.


\subsection{A characterization of adaptive approximation classes}

In this subsection, we give necessary and sufficient conditions for $u\in H^1_\Gamma$ to be in $\tstA^s(\rho_d)$.
These conditions will be in terms of memberships of $u$ and $Tu$ into suitable approximation classes,
which, in light of the preceding section, are related to Besov spaces.
The coefficients of $T$ are required to satisfy conditions of the form $g\in\bar\tstA^s_{\infty;d}(\Leb[\theta]{\infty}(\Omega))$,
where the latter space is defined in \eqref{e:adapt-class-disc},
and again these spaces can be cast in terms of Besov spaces with the help of Theorem \ref{t:direct-Lp-disc} and Theorem \ref{t:multilevel-Besov-disc}.

\begin{theorem}\label{t:elliptic-2nd}
Let $s>0$ and let $d\geq m-2$. 
Assume that $a_{ij}\in\bar\tstA^s_{\infty;d+2-m}(\Leb{\infty}(\Omega))$, $b_{i}\in\bar\tstA^s_{\infty;d+1-m}(\Leb[1]{\infty}(\Omega))$, and $c\in\bar\tstA^s_{\infty;d-m}(\Leb[2]{\infty}(\Omega))$.
Then we have
\begin{equation}
\tstA^{s}(\rho_d) = \tstA^{s}(H^1_\Gamma)\cap T^{-1}(\bar\tstA^s_{\infty;d}(\Leb[1]{2}(\Omega))) .
\end{equation}
\end{theorem}

The proof of this theorem will be given below in Lemma \ref{l:elliptic-2nd-direct} and Lemma \ref{l:elliptic-2nd-inverse}.
Before proving those lemmata, let us make a few points on the conditions of the theorem.

First, recall from Theorem \ref{t:direct-Lp-disc} that $\bar A^{\sigma}_{q,q;d}(\Omega)\subset\bar\tstA^s_{\infty;d}(\Leb[1]{2}(\Omega))$
for $\sigma=sn-1$ and $0\leq\frac1q - \frac12\leq\frac\sigma{n}$, 
and from Theorem \ref{t:multilevel-Besov-disc} that $B^{\sigma}_{q,q}(\Omega)\subset\bar A^{\sigma}_{q,q;d}(\Omega)$ for $\sigma<d+\max\{1,\frac1q\}$.

Second, while the approximation classes $\tstA^{s}(H^1_\Gamma)$ are associated to the finite element spaces $S_P=H^1_\Gamma\cap S^m_P$,
the approximation classes we considered in the preceding section are associated to the spaces $S^m_P$ with no boundary conditions.
In view of applying Theorem \ref{t:direct-App} and Theorem \ref{t:multilevel-Besov}, we need the latter type of approximation classes.
The following lemma provides a link between the two types.

\begin{lemma}
For $s>0$ we have
\begin{equation}\label{e:H1G}
\tstA^{s}(H^1_\Gamma)\equiv\tstA^{s}(H^1,\tstP,\{H^1_\Gamma\cap S^m_P\})=H^1_\Gamma\cap\tstA^{s}(A^1_{2,2}(\Omega),\tstP,\{S^m_P\}) .
\end{equation}
In particular, we have $H^1_\Gamma\cap A^\alpha_{p,p}(\Omega)\subset\tstA^{s}(H^1_\Gamma)$ for $\alpha=sn+1$ and $\frac1p<s+\frac12$.
\end{lemma}

\begin{proof}
Let $u\in H^1_\Gamma$, and let $u_P\in H^1_\Gamma\cap  S^m_P$ be the Scott-Zhang interpolator of $u$ adapted to the Dirichlet boundary condition on $\Gamma$, cf. \cite{SZ90}.
We have
\begin{equation}
\inf_{v\in H^1_\Gamma\cap S^m_P} \|u-v\|_{H^1(\Omega)}
\leq \|u-u_P\|_{H^1(\Omega)}
\lesssim \inf_{v\in S^m_P} \|u-v\|_{H^1(\Omega)} ,
\end{equation}
by standard properties of the Scott-Zhang interpolator.
Since $H^1=B^1_{2,2}=A^1_{2,2}$ by Theorem \ref{t:multilevel-Besov}, 
this implies \eqref{e:H1G}.
Then the second assertion of the theorem follows from a direct application of Theorem \ref{t:direct-App}.
\end{proof}

The inclusion $\tstA^{s}(H^1_\Gamma)\cap T^{-1}(\bar\tstA^s_{\infty;d}(\Leb[1]{2}(\Omega)))\subset\tstA^{s}(\rho_d)$
of Theorem \ref{t:elliptic-2nd} is a consequence of the following lemma.

\begin{lemma}\label{l:elliptic-2nd-direct}
For $u\in H^1_\Gamma$ and $P\in\tstP$, there exists $v\in S_P$ such that
\begin{equation}\label{e:elliptic-2nd-direct}
\begin{split}
&\rho_d(u,v,P)^2 
\lesssim E(u,S_P)_{H^1(\Omega)}^2 
+ E(Tu,\bar S^d_P)_{\Leb[1]{2}(\Omega)}^2 \\
&\quad + 
\left( E(a_{ij},\bar S^{d+2-m}_P)_{\Leb{\infty}(\Omega)}^2 
+ E(b_{i},\bar S^{d+1-m}_P)_{\Leb[1]{\infty}(\Omega)}^2 
+ E(c,\bar S^{d-m}_P)_{\Leb[2]{\infty}(\Omega)}^2 \right) 
|u|_{H^1(\Omega)}^2 .
\end{split}
\end{equation}
\end{lemma}

\begin{proof}
We take $v$ to be the Scott-Zhang interpolator of $u$ adapted to the Dirichlet boundary condition on $\Gamma$, cf. \cite{SZ90}.
We have
\begin{equation}
\|u-v\|_{H^1} \lesssim \inf_{w\in S_P} \|u-w\|_{H^1},
\end{equation}
for all $P\in\tstP$.
It remains to bound the oscillation term.

First, let us consider the special case where the coefficients of $T$ are piecewise polynomials subordinate to $P$.
More specifically, assume that $a_{ij}|_\tau\in\Pol_{d+2-m}$, $b_{i}|_\tau\in\Pol_{d+1-m}$, and $c|_\tau\in\Pol_{d-m}$ for each $\tau\in P$.
In this case, the oscillations associated to edges vanish, because the edge residuals $r_e$ are polynomials of degree not exceeding $d+1$.
For the element residuals, with the shorthand $f=Tu$, we have
\begin{equation}\label{e:element-residual-triangle}
\|(1-\Pi_\tau)(f-Tv)\|_{\Leb{2}(\tau)} \leq \|(1-\Pi_\tau)f\|_{\Leb{2}(\tau)} + \|(1-\Pi_\tau)Tv\|_{\Leb{2}(\tau)},
\end{equation}
and the last term is zero because $Tv\in \Pol_{d}$.
The remaining term gives rise to
\begin{equation}
\sum_{\tau\in P}h_\tau^2\|(1-\Pi_\tau)f\|_{\Leb{2}(\tau)}^2 = E(f,\bar S^d_P)_{\Leb[1]{2}(\Omega)}^2 ,
\end{equation}
which yields the desired result.

In the general case, the edge residuals and the terms $Tv|_\tau$ can be nonpolynomial.
Let us treat $Tv|_\tau=- a_{ij} \partial_i \partial_j v + b_{k} \partial_k v + c v$ term by term.
We have
\begin{equation}
(1-\Pi_\tau)a_{ij}\partial_i\partial_jv = (1-\Pi_\tau)(a_{ij} - \bar a_{ij})\partial_i\partial_jv,
\qquad \bar a_{ij} \in \Pol_{d+2-m} ,
\end{equation}
which implies
\begin{equation}
\|(1-\Pi_\tau)a_{ij}\partial_i\partial_jv \|_{\Leb{2}(\tau)} 
= \|(a_{ij} - \bar a_{ij})\partial_i\partial_jv \|_{\Leb{2}(\tau)}
\leq \|a_{ij} - \bar a_{ij}\|_{\Leb{\infty}(\tau)} \|\partial_i\partial_jv \|_{\Leb{2}(\tau)} ,
\end{equation}
for any $\bar a_{ij} \in \Pol_{d+2-m}$.
Now we think of $\bar a_{ij}$ as a function in $\bar S^{d+2-m}_P$ that approximates $a_{ij}$ in each element $\tau\in P$ with the best $\Leb{\infty}(\tau)$-error.
As a result, we get
\begin{equation}\label{e:elliptic-2nd-direct-aij}
\begin{split}
\sum_{\tau\in P} h_\tau^2 \|(1-\Pi_\tau)a_{ij}\partial_i\partial_jv \|_{\Leb{2}(\tau)}^2 
&\leq \sum_{\tau\in P} h_\tau^2 \|a_{ij} - \bar a_{ij}\|_{\Leb{\infty}(\tau)}^2 \|\partial_i\partial_jv \|_{\Leb{2}(\tau)}^2 \\
&\leq E(a_{ij},\bar S^{d+2-m}_P)_{\Leb{\infty}(\Omega)}^2 \sum_{\tau\in P} h_\tau^2 \|\partial_i\partial_jv \|_{\Leb{2}(\tau)}^2 \\
&\lesssim E(a_{ij},\bar S^{d+2-m}_P)_{\Leb{\infty}(\Omega)}^2 \|\nabla v \|_{\Leb{2}(\Omega)}^2 ,
\end{split}
\end{equation}
where we have used an inverse inequality in the last step.
In light of the $H^1$-stability of the Scott-Zhang projector,
this is one of the terms in the right hand side of \eqref{e:elliptic-2nd-direct}.

Similarly, let $\bar c$ be a function in $\bar S^{d-m}_P$ that approximates $c$ in each element $\tau\in P$ with the best $\Leb{\infty}(\tau)$-error.
Then we have
\begin{equation}\label{e:elliptic-2nd-direct-c}
(1-\Pi_\tau) c v = (1-\Pi_\tau) (c - \bar c) v = (1-\Pi_\tau) (c - \bar c) (v - \bar v),
\end{equation}
in each $\tau\in P$, where $\bar v$ is the average of $v$ over $\tau$.
This yields
\begin{equation}
\begin{split}
\sum_{\tau\in P} h_\tau^2 \|(1-\Pi_\tau)cv\|_{\Leb{2}(\tau)}^2 
&\leq \sum_{\tau\in P} h_\tau^2 \|c - \bar c\|_{\Leb{\infty}(\tau)}^2 \|v - \bar v\|_{\Leb{2}(\tau)}^2 \\
&\lesssim \sum_{\tau\in P} h_\tau^4 \|c - \bar c\|_{\Leb{\infty}(\tau)}^2 \|\nabla v\|_{\Leb{2}(\tau)}^2 \\
&\leq E(c,\bar S^{d-m}_P)_{\Leb[2]{\infty}(\Omega)}^2 \|\nabla v\|_{\Leb{2}(\Omega)}^2 ,
\end{split}
\end{equation}
where we have used the Poincar\'e inequality in the second line.
Estimation of the term involving $b_i\partial_iv$ is more straightforward, which we omit.

As for the edge oscillations, let $\tau\in P$, and let $e$ be an edge of $\tau$.
Then we have
\begin{equation}\label{e:elliptic-2nd-direct-edge}
\begin{split}
\|(1-\Pi_e)a_{ij}\partial_jv\|_{\Leb{2}(e)} 
&= \|(1-\Pi_e)(a_{ij}-\bar a_{ij})\partial_jv\|_{\Leb{2}(e)} \\
&\leq \|(a_{ij}-\bar a_{ij})\partial_jv\|_{\Leb{2}(e)} \\
&\leq \|a_{ij}-\bar a_{ij}\|_{\Leb{\infty}(e)} \|\partial_jv\|_{\Leb{2}(e)} \\
&\lesssim h_e^{-\frac12} \|a_{ij}-\bar a_{ij}\|_{\Leb{\infty}(\tau)} \|\nabla v\|_{\Leb{2}(\tau)} ,
\end{split}
\end{equation}
for any $\bar a_{ij}\in\Pol_{d+2-m}$,
which shows that 
the contribution of the edge oscillations to the final estimate \eqref{e:elliptic-2nd-direct}
is identical to that of \eqref{e:elliptic-2nd-direct-aij}.
\end{proof}

\begin{remark}
By using the fact that the Scott-Zhang projector is bounded in $H^t(\Omega)$ for $t<\frac32$,
we could have introduced extra powers of $h_\tau$ or $h_e$ into the estimates \eqref{e:elliptic-2nd-direct-aij}, \eqref{e:elliptic-2nd-direct-c}, and \eqref{e:elliptic-2nd-direct-edge}.
This means that the regularity conditions on the coefficients $a_{ij}$, $b_k$, and $c$ in Theorem \ref{t:elliptic-2nd} can be relaxed slightly,
if the conclusion of the theorem is to be changed to $\tstA^{s}(H^1_\Gamma)\cap H^t(\Omega)\cap T^{-1}(\bar\tstA^s_{\infty;d}(\Leb[1]{2}(\Omega)))\subset\tstA^{s}(\rho_d)$ with $1<t<\frac32$.
\end{remark}

\begin{lemma}\label{l:elliptic-2nd-inverse}
For any $u\in H^1_\Gamma$, $P\in\tstP$ and $v\in S_P$, we have
\begin{equation}\label{e:elliptic-2nd-inverse}
\begin{split}
&E(Tu,\bar S^d_P)_{\Leb[1]{2}(\Omega)}^2 + \|u-v\|_{H^1(\Omega)}^2
\lesssim  \rho_d(u,v,P)^2 \\
&\quad + 
\left( E(a_{ij},\bar S^{d+2-m}_P)_{\Leb{\infty}(\Omega)}^2 
+ E(b_{i},\bar S^{d+1-m}_P)_{\Leb[1]{\infty}(\Omega)}^2 
+ E(c,\bar S^{d-m}_P)_{\Leb[2]{\infty}(\Omega)}^2 \right) 
|v|_{H^1(\Omega)}^2 .
\end{split}
\end{equation}
In particular, under the hypotheses of Theorem \ref{t:elliptic-2nd}, we have the inclusion $\tstA^{s}(\rho_d)\subset\tstA^{s}(H^1_\Gamma)\cap T^{-1}(\bar\tstA^s_{\infty;d}(\Leb[1]{2}(\Omega)))$.
\end{lemma}

\begin{proof}
All the ingredients for establishing the estimate \eqref{e:elliptic-2nd-inverse} is already given in the proof of the preceding lemma.
Namely, we start with the bound
\begin{equation}
\begin{split}
(\osc_d(u,v,P))^2 
&\lesssim \sum_{\tau\in P} h_\tau^{2} \|(1-\Pi_\tau)Tu\|_{\Leb{2}(\tau)}^2 + \sum_{\tau\in P} h_\tau^{2} \|(1-\Pi_\tau)Tv\|_{\Leb{2}(\tau)}^2 \\
&\quad + \sum_{e\in E_P} h_e \|(1-\Pi_e) r_e\|_{\Leb{2}(e)}^2 ,
\end{split}
\end{equation}
and use the estimates \eqref{e:elliptic-2nd-direct-aij}, \eqref{e:elliptic-2nd-direct-c}, and \eqref{e:elliptic-2nd-direct-edge}, etc.,
on the last two terms to get \eqref{e:elliptic-2nd-inverse}.

As for the second assertion, let $\{P_k\}\subset\tstP$ and $\{v_k\}$ be two sequences with $v_k\in S_{P_k}$
such that $\#P_k\lesssim2^k$ and $\rho_d(u,v_k,P_k)\lesssim2^{-ks}$.
Then since $\|u-v_k\|_{H^1}\leq\rho_d(u,v_k,P_k)$, we have $\|v_k\|_{H^1}\lesssim\|u\|_{H^1}$.
Hence, by employing overlay of partitions, without loss of generality, we can suppose that the right hand side of \eqref{e:elliptic-2nd-inverse} with $P=P_k$ and $v=v_k$ is bounded by 
a constant multiple of $2^{-ks}$.
Looking at the left hand side then reveals that $Tu\in\tstA^s_{\infty;d}(\Leb[1]{2}(\Omega))$ and $u\in\tstA^s(H^1_\Gamma)$.
\end{proof}

\section*{Acknowledgements}

This work was supported by NSERC Discovery Grants Program and FQRNT New University Researchers Start Up Program.




\bibliographystyle{plainnat}
\bibliography{../bib/timur}

\begin{thebibliography}{20}
\providecommand{\natexlab}[1]{#1}
\providecommand{\url}[1]{\texttt{#1}}
\expandafter\ifx\csname urlstyle\endcsname\relax
  \providecommand{\doi}[1]{doi: #1}\else
  \providecommand{\doi}{doi: \begingroup \urlstyle{rm}\Url}\fi

\bibitem[Bergh and L{{\"o}}fstr{{\"o}}m(1976)]{BL76}
J{{\"o}}ran Bergh and J{{\"o}}rgen L{{\"o}}fstr{{\"o}}m.
\newblock \emph{Interpolation spaces. {A}n introduction}.
\newblock Springer-Verlag, Berlin, 1976.
\newblock Grundlehren der Mathematischen Wissenschaften, No. 223.

\bibitem[Binev et~al.(2002)Binev, Dahmen, DeVore, and Petrushev]{BDDP02}
Peter Binev, Wolfgang Dahmen, Ronald DeVore, and Pencho Petrushev.
\newblock Approximation classes for adaptive methods.
\newblock \emph{Serdica Math. J.}, 28\penalty0 (4):\penalty0 391--416, 2002.
\newblock Dedicated to the memory of Vassil Popov on the occasion of his 60th
  birthday.

\bibitem[Binev et~al.(2004)Binev, Dahmen, and DeVore]{BDD04}
Peter Binev, Wolfgang Dahmen, and Ron DeVore.
\newblock Adaptive finite element methods with convergence rates.
\newblock \emph{Numer. Math.}, 97\penalty0 (2):\penalty0 219--268, 2004.
\newblock URL \url{http://dx.doi.org/10.1007/s00211-003-0492-7}.

\bibitem[Birman and Solomyak(1967)]{BirSol67}
Mikhail~Shlemovich Birman and Mikhail~Zakharovich Solomyak.
\newblock Piecewise polynomial approximations of functions of classes
  {$W^\alpha_p$}.
\newblock \emph{Mat. Sb. (N.S.)}, 73\penalty0 (3):\penalty0 331--355, 1967.

\bibitem[Bonito and Nochetto(2010)]{BN10}
Andrea Bonito and Ricardo~H. Nochetto.
\newblock Quasi-optimal convergence rate of an adaptive discontinuous
  {G}alerkin method.
\newblock \emph{SIAM J. Numer. Anal.}, 48\penalty0 (2):\penalty0 734--771,
  2010.
\newblock ISSN 0036-1429.
\newblock \doi{10.1137/08072838X}.
\newblock URL \url{http://dx.doi.org/10.1137/08072838X}.

\bibitem[Cascon et~al.(2008)Cascon, Kreuzer, Nochetto, and Siebert]{CKNS08}
J.~Manuel Cascon, Christian Kreuzer, Ricardo~H. Nochetto, and Kunibert~G.
  Siebert.
\newblock Quasi-optimal convergence rate for an adaptive finite element method.
\newblock \emph{SIAM J. Numer. Anal.}, 46\penalty0 (5):\penalty0 2524--2550,
  2008.
\newblock URL \url{http://dx.doi.org/10.1137/07069047X}.

\bibitem[Cohen et~al.(2001)Cohen, Dahmen, and DeVore]{CDD01}
Albert Cohen, Wolfgang Dahmen, and Ronald DeVore.
\newblock Adaptive wavelet methods for elliptic operator equations: convergence
  rates.
\newblock \emph{Math. Comp.}, 70\penalty0 (233):\penalty0 27--75, 2001.
\newblock ISSN 0025-5718.
\newblock \doi{10.1090/S0025-5718-00-01252-7}.
\newblock URL \url{http://dx.doi.org/10.1090/S0025-5718-00-01252-7}.

\bibitem[Dekel and Leviatan(2004)]{DL04a}
Shai Dekel and Dany Leviatan.
\newblock Whitney estimates for convex domains with applications to
  multivariate piecewise polynomial approximation.
\newblock \emph{Found. Comput. Math.}, 4\penalty0 (4):\penalty0 345--368, 2004.
\newblock ISSN 1615-3375.
\newblock \doi{10.1007/s10208-004-0096-3}.
\newblock URL \url{http://dx.doi.org/10.1007/s10208-004-0096-3}.

\bibitem[DeVore and Lorentz(1993)]{DL93}
Ronald~A. DeVore and George~G. Lorentz.
\newblock \emph{Constructive approximation}, volume 303 of \emph{Grundlehren
  der Mathematischen Wissenschaften [Fundamental Principles of Mathematical
  Sciences]}.
\newblock Springer-Verlag, Berlin, 1993.
\newblock ISBN 3-540-50627-6.

\bibitem[D{\"o}rfler(1996)]{Dorf96}
Willy D{\"o}rfler.
\newblock A convergent adaptive algorithm for {P}oisson's equation.
\newblock \emph{SIAM J. Numer. Anal.}, 33\penalty0 (3):\penalty0 1106--1124,
  1996.
\newblock URL \url{http://dx.doi.org/10.1137/0733054}.

\bibitem[Feischl et~al.(2014)Feischl, F{\"u}hrer, and Praetorius]{FFP14}
Michael Feischl, Thomas F{\"u}hrer, and Dirk Praetorius.
\newblock Adaptive {FEM} with optimal convergence rates for a certain class of
  nonsymmetric and possibly nonlinear problems.
\newblock \emph{SIAM J. Numer. Anal.}, 52\penalty0 (2):\penalty0 601--625,
  2014.
\newblock ISSN 0036-1429.
\newblock \doi{10.1137/120897225}.
\newblock URL \url{http://dx.doi.org/10.1137/120897225}.

\bibitem[Gantumur et~al.(2007)Gantumur, Harbrecht, and Stevenson]{GHS07}
{Ts}ogtgerel Gantumur, Helmut Harbrecht, and Rob Stevenson.
\newblock An optimal adaptive wavelet method without coarsening of the
  iterands.
\newblock \emph{Math. Comp.}, 76\penalty0 (258):\penalty0 615--629, 2007.
\newblock ISSN 0025-5718.
\newblock \doi{10.1090/S0025-5718-06-01917-X}.
\newblock URL \url{http://dx.doi.org/10.1090/S0025-5718-06-01917-X}.

\bibitem[Gaspoz and Morin(2013)]{GM13}
Fernando~D. Gaspoz and Pedro Morin.
\newblock Approximation classes for adaptive higher order finite element
  approximation.
\newblock \emph{Math. Comp.}, 2013.
\newblock To appear.

\bibitem[Morin et~al.(2000)Morin, Nochetto, and Siebert]{MNS00}
Pedro Morin, Ricardo~H. Nochetto, and Kunibert~G. Siebert.
\newblock Data oscillation and convergence of adaptive {FEM}.
\newblock \emph{SIAM J. Numer. Anal.}, 38\penalty0 (2):\penalty0 466--488
  (electronic), 2000.
\newblock URL \url{http://dx.doi.org/10.1137/S0036142999360044}.

\bibitem[Nochetto et~al.(2009)Nochetto, Siebert, and Veeser]{NSV09}
Ricardo~H. Nochetto, Kunibert~G. Siebert, and Andreas Veeser.
\newblock Theory of adaptive finite element methods: an introduction.
\newblock In \emph{Multiscale, nonlinear and adaptive approximation}, pages
  409--542. Springer, Berlin, 2009.
\newblock URL \url{http://dx.doi.org/10.1007/978-3-642-03413-8_12}.

\bibitem[Oswald(1994)]{Osw94}
Peter Oswald.
\newblock \emph{Multilevel finite element approximation}.
\newblock Teubner Skripten zur Numerik. [Teubner Scripts on Numerical
  Mathematics]. B. G. Teubner, Stuttgart, 1994.
\newblock ISBN 3-519-02719-4.
\newblock Theory and applications.

\bibitem[Pietsch(1981)]{Piet81}
Albrecht Pietsch.
\newblock Approximation spaces.
\newblock \emph{J. Approx. Theory}, 32\penalty0 (2):\penalty0 115--134, 1981.
\newblock ISSN 0021-9045.
\newblock \doi{10.1016/0021-9045(81)90109-X}.
\newblock URL \url{http://dx.doi.org/10.1016/0021-9045(81)90109-X}.

\bibitem[Scott and Zhang(1990)]{SZ90}
L.~Ridgway Scott and Shangyou Zhang.
\newblock Finite element interpolation of nonsmooth functions satisfying
  boundary conditions.
\newblock \emph{Math. Comp.}, 54\penalty0 (190):\penalty0 483--493, 1990.
\newblock URL \url{http://dx.doi.org/10.2307/2008497}.

\bibitem[Stevenson(2007)]{Stev07}
Rob Stevenson.
\newblock Optimality of a standard adaptive finite element method.
\newblock \emph{Found. Comput. Math.}, 7\penalty0 (2):\penalty0 245--269, 2007.
\newblock URL \url{http://dx.doi.org/10.1007/s10208-005-0183-0}.

\bibitem[Stevenson(2008)]{Stev08}
Rob Stevenson.
\newblock The completion of locally refined simplicial partitions created by
  bisection.
\newblock \emph{Math. Comp.}, 77\penalty0 (261):\penalty0 227--241, 2008.
\newblock URL \url{http://dx.doi.org/10.1090/S0025-5718-07-01959-X}.

\end{thebibliography}

\end{document}